\documentclass[11pt,twoside]{article}

\usepackage[margin=1in]{geometry} 
\usepackage{amsmath,amsthm,amssymb,bm}
\usepackage{color} 
\usepackage{mathtools,mathptmx}
\usepackage{indentfirst}
\usepackage{mathptmx}

\usepackage{amsbsy}  
\usepackage{amsfonts}
\usepackage{graphicx}
\usepackage{tkz-euclide}

\usepackage{array}
\usepackage{cases}

\usepackage{hyperref}
\hypersetup{colorlinks=true, allcolors=blue}

\usepackage[square, numbers, comma, sort&compress]{natbib}  
\usepackage[all]{xy}
\usepackage{blkarray}
\usepackage{datetime}
\usepackage{enumitem}    
\usepackage{relsize}
\usepackage[super]{nth}
\usepackage{imakeidx}

\usepackage{setspace}

\newtheorem{theorem}{Theorem}[section]
\newtheorem{proposition}[theorem]{Proposition}
\newtheorem{corollary}[theorem]{Corollary}
\newtheorem{lemma}[theorem]{Lemma}

\newtheorem{remark}[theorem]{Remark}

\numberwithin{equation}{section}
\numberwithin{theorem}{section}

\newcommand{\dis}{\displaystyle}
\newcommand{\R}{\mathbb{R}}
\newcommand{\C}{\mathbb{C}}
\newcommand{\Z}{\mathbb{Z}}
\newcommand{\N}{\mathbb{N}}

\newcommand{\Real}{\operatorname{Re}}

\newcommand{\pochhammer}[2][n]{\left(#2\right)_{#1}}

\newcommand{\Hypergeometric}[5][x]{{}_{#2} F_{#3} \left( \left.\begin{matrix} #4 \\ #5 \end{matrix} \, \right| \, #1\right)}

\newcommand{\MeijerG}[5][x]{G_{#3}^{\,#2}\left(\left.\begin{matrix} #4 \\ #5 \end{matrix} \, \right| \, #1\right)}

\newcommand{\ceil}[1]{\left\lceil #1 \right\rceil}
\newcommand{\seq}[2][n\in\N]{\left(#2\right)_{#1}}

\newcommand{\n}{\vec{n}}

\newcommand{\e}{\mathrm{e}}

\newcommand{\Functional}[2]{#1\left[#2\right]}

\newcommand{\StieltjesRogersPoly}[3][r]{S_{#2}^{\,(#1)}\left(#3\right)}
\newcommand{\modifiedStieltjesRogersPoly}[4][r]{S_{#2}^{\,(#1;\,#3)}\left(#4\right)}

\newcommand{\generalisedStieltjesRogersPolyTypeJ}[5][r]{S_{#2,#3}^{\,(#1;\,#4)}\left(#5\right)}

\newcommand{\JacobiRogersPoly}[3][r]{J_{#2}^{\,(#1)}\left(#3\right)}
\newcommand{\generalisedJacobiRogersPoly}[4][r]{J_{#2,#3}^{\,(#1)}\left(#4\right)}

\setlist[itemize]{leftmargin=*,topsep=0pt,itemsep=6pt}
\setlist[enumerate]{leftmargin=*,topsep=0pt,itemsep=6pt}

\begin{document}
\title{Bidiagonal matrix factorisations associated with symmetric multiple orthogonal polynomials and lattice paths}
\author{H\'elder Lima
\thanks{CIDMA, Departamento de Matemática, Universidade de Aveiro, 3810-193 Aveiro, Portugal
\\ \hspace*{0.5 cm} helder.lima@ua.pt }}
\date{Published in the OPSFA17 special issue of \textit{Numerical Algorithms}}
\maketitle
\vspace*{-0.75 cm}
\thispagestyle{empty}
	
\begin{abstract}
The central object of study in this paper are infinite banded Hessenberg matrices admitting factorisations as products of bidiagonal matrices. 
In the two main novel results of this paper, we show that these Hessenberg matrices are associated with the decomposition of $(r+1)$-fold symmetric $r$-orthogonal polynomials and are the production matrices of the generating polynomials of $r$-Dyck paths.

We combine the aforementioned bidiagonal matrix factorisations and the recently found connection of multiple orthogonal polynomials with lattice paths and branched continued fractions to study $(r+1)$-fold symmetric $r$-orthogonal polynomials on a star-like set of the complex plane and their decomposition via multiple orthogonal polynomials on the positive real line.
As an explicit example, we give formulas as terminating hypergeometric series for the Appell sequences of $(r+1)$-fold symmetric $r$-orthogonal polynomials on a star-like set and show that the densities of their orthogonality measures can be expressed via Meijer G-functions on the positive real line.
\end{abstract}

\noindent\textbf{Keywords:} 
\textit{Hessenberg matrices, bidiagonal matrices, multiple orthogonal polynomials, $m$-fold symmetry, production matrices, lattice paths.}

\noindent\textbf{MSC2020:} 
15A23, 33C45, 42C05 (primary); 05A19, 15B48, 30C15, 30E05, 33C20 (secondary).

\noindent\textbf{Acknowledgements:}
I am very grateful to Walter Van Assche for several enlightening discussions about the investigation presented here and for numerous pertinent suggestions that considerably improved this paper, to Alan Sokal for kindly sharing with me a draft version of \cite{AlanEtAl-LPandBCF3}, to Arno Kuijlaars for drawing my attention to \cite{DelvauxLopezHighOrder3TermRec}, and to Ana Foulqui\'e Moreno for helpful comments about some of the changes made from the first available version of this paper.

\noindent\textbf{Funding:}
This research was supported by the Flemish Research Foundation (FWO) project G0C9819N, the Belgian ``Excellence of Science" project PRIMA 30889451, and the Portuguese Foundation for Science and Technology (FCT) projects 10.54499/UIDB/04106/2020 and 10.54499/UIDP/04106/2020.



\section{Introduction}
The central object of study in this paper are infinite matrices $\mathrm{H}^{(r;j)}$, with $r\in\Z^+$ and $0\leq j\leq r$, admitting factorisations 
\begin{equation}
\label{products of lower bidiagonal matrices with an upper bidiagonal matrix in the middle}
\mathrm{H}^{(r;j)}=\mathrm{L}_{j+1}\cdots\mathrm{L}_r\,\mathrm{U}\,\mathrm{L}_1\cdots\mathrm{L}_j,
\end{equation} 
as products of $r$ lower-bidiagonal matrices $\mathrm{L}_1,\cdots,\mathrm{L}_r$ and an upper-bidiagonal matrix $\mathrm{U}$ of the form 
\begin{equation}
\label{bidiagonal matrices U and L_k}
\mathrm{L}_k=
\begin{bmatrix} 
1\\
\alpha_k & 1  \\
& \alpha_{r+1+k} & 1  \\
& & \ddots & \ddots
\end{bmatrix}
\text{ for }1\leq k\leq r,
\quad\text{and}\quad
\mathrm{U}=
\begin{bmatrix} 
\alpha_0 & 1  \\
& \alpha_{r+1} & 1  \\
& & \alpha_{2(r+1)} & 1 \\ 
& & & \ddots & \ddots
\end{bmatrix},
\end{equation}
whose nontrivial entries are given by a sequence $\seq[n\in\N]{\alpha_n}$. 
The matrices $\mathrm{H}^{(r;j)}$ are $(r+2)$-banded unit-lower-Hessenberg matrices, i.e., all their nonzero entries are located on the supradiagonal, the main diagonal, and the first $r$ subdiagonals, with all entries on the supradiagonal equal to $1$.

(Positive) bidiagonal factorisations of Hessenberg matrices (or, more generally, banded matrices) associated with (mixed) multiple orthogonal polynomials is a timely topic of research, as it can be seen in \cite{AnaAmilcarManuel-SpectralTheory...,AnaAmilcarManuel-Oscillatory...,AnaAmilcarManuel-Positive...}.

In the two main novel results of this paper, both presented in Section \ref{Bidiagonal factorisations for Hessenberg matrices}, we show that the banded Hessenberg matrices satisfying the bidiagonal factorisations \eqref{products of lower bidiagonal matrices with an upper bidiagonal matrix in the middle} are the recurrence matrices for the components of the decomposition of $(r+1)$-fold symmetric $r$-orthogonal polynomials and are the production matrices of the generating polynomials of partial $r$-Dyck paths. 
Using these results, we can apply the connection of multiple orthogonal polynomials with lattice paths and branched continued fractions, introduced in \cite{AlanSokal.MOP-ProdMat-BCF} and further explored in \cite{HypergeometricMOP+BCF}, to study $(r+1)$-fold symmetric $r$-orthogonal polynomials.
In Section \ref{Symmetric MOP}, we revisit known results about the location of the zeros of these polynomials and their multiple orthogonality with respect to measures supported on a star-like set of the complex plane, give combinatorial interpretations for the moments of their dual sequence, find orthogonality measures on the positive real line for the components of their decomposition, and present an explicit example of $(r+1)$-fold symmetric $r$-orthogonal polynomials. The latter form Appell polynomial sequences, can be represented as terminating hypergeometric series, and satisfy orthogonality conditions with respect to measures that can be expressed via Meijer G-functions.

\section{Background}
\label{Background}
\subsection{Multiple orthogonality and $m$-fold symmetry}
\label{Background MOP}
Multiple orthogonal polynomials are a generalisation of conventional orthogonal polynomials, satisfying orthogonality relations with respect to several measures instead of just one.
We give a brief introduction to this topic, presenting some basic definitions and notation that we use throughout the paper.
See \cite[Ch.~23]{IsmailBook}, \cite[Ch.~4]{NikishinSorokinBook}, and \cite[\S 3]{WalterSurveyAIMS2018} for more information about multiple orthogonal polynomials.

There are two types of multiple orthogonal polynomials: type I and type II. 
Both type I and type II polynomials satisfy orthogonality conditions with respect to a system of $r$ measures (or linear functionals), for a positive integer $r$, and depend on a multi-index $\n=\left(n_1,\cdots,n_r\right)\in\N^r$ of norm $|\n|=n_1+\cdots+n_r$. 
Here and always throughout this paper, we consider $0\in\N$.
Both type I and type II polynomials reduce to conventional orthogonal polynomials when the number of orthogonality measures (or functionals) is equal to $1$.

In this paper we focus only on type II multiple orthogonal polynomials for multi-indices on the so-called step-line.
A multi-index $\left(n_1,\cdots,n_r\right)\in\N^r$ is on the \textit{step-line} if $n_1\geq n_2\geq\cdots\geq n_r\geq n_1-1$.
For a fixed $r\in\Z^+$, there is a unique multi-index of norm $n$ on the step-line of $\N^r$ for each $n\in\N$: 
if $n=rk+j$ with $k\in\N$ and $0\leq j\leq r-1$, then the multi-index of norm $n$ on the step-line is of the form
$\vec{n}=\left(k+1,\cdots,k+1,k,\cdots,k\right)$, with $j$ entries equal to $k+1$ and $r-j$ entries equal to $k$.

The degree of a type II multiple orthogonal polynomial is equal to the norm of its multi-index. 
Therefore, the type II polynomials on the step-line form a sequence with exactly one polynomial of degree $n$, for each $n\in\N$, and we can replace the multi-index by the degree of the polynomial without any ambiguity.

The type II multiple orthogonal polynomials on the step-line of $\N^r$ are often referred to as $r$-orthogonal polynomials \cite{MaroniOrthogonalite}, where $r$ is the number of orthogonality measures (or functionals).
A polynomial sequence $\seq{P_n(x)}$ is \textit{$r$-orthogonal}, i.e., $\seq{P_n(x)}$ is the sequence of type II multiple orthogonal polynomials on the step-line, with respect to a system of measures $\left(\mu_1,\cdots,\mu_r\right)$ if it satisfies
\begin{equation}
\label{d-orthogonality conditions}
\int x^kP_n(x)\mathrm{d}\mu_j(x)
=\begin{cases}
N_n\neq 0 & \text{ if } n=rk+j-1, \\
  \hfil 0 & \text{ if } n\geq rk+j,
\end{cases}
\quad\text{for all }1\leq j\leq r.
\end{equation}
Multiple orthogonality conditions can also be written using linear functionals in a generic field $\mathbb{K}$ instead of measures on $\R$ or $\C$.
A linear functional on $\mathbb{K}[x]$ is a linear map $u:\mathbb{K}[x]\to\mathbb{K}$.
The action of a linear functional $u$ on a polynomial $f\in\mathbb{K}[x]$ is denoted by $\Functional{u}{f}$.
For $n\in\N$, the moment of order $n$ of a linear functional $u$ is $\Functional{u}{x^n}$.
By linearity, every linear functional $u$ is uniquely determined by its moments. 
A polynomial sequence $\seq{P_n(x)}$ is $r$-orthogonal with respect to a system of linear functionals $\left(v_1,\cdots,v_r\right)$ if \eqref{d-orthogonality conditions} holds with $\dis\int{x^kP_n(x)\mathrm{d}\mu_j(x)}$ replaced by $\Functional{v_j}{x^k\,P_n(x)}$ for all values of $k,n,j$. 

A sequence of linear functionals $\seq[k\in\N]{u_k}$ is the \textit{dual sequence} of a polynomial sequence $\seq{P_n(x)}$ if 
\begin{equation}
\label{dual sequence}
\Functional{u_k}{P_n}=\delta_{k,n}=
\begin{cases} 
1 & \text{if }k=n, \\ 
0 & \text{if }k\neq n. 
\end{cases}
\end{equation}
If $\seq{P_n(x)}$ is a $r$-orthogonal polynomial sequence with dual sequence $\seq[k\in\N]{u_k}$, then $\seq{P_n(x)}$ is $r$-orthogonal with respect to $\left(u_0,\cdots,u_{r-1}\right)$. 
Based on \cite[Thm.~2.1]{MaroniOrthogonalite}, a polynomial sequence $\seq{P_n(x)}$ is $r$-orthogonal if and only if it satisfies a $(r+1)$-order recurrence relation of the form
\begin{equation}
\label{recurrence relation for a r-OPS}
P_{n+1}(x)=x\,P_n(x)-\sum_{k=0}^{\min(n,r)}\gamma_{n-k}^{\,[k]}\,P_{n-k}(x).
\end{equation}

The recurrence coefficients in \eqref{recurrence relation for a r-OPS} are collected in the infinite $(r+2)$-banded unit-lower-Hessenberg matrix $\mathrm{H}=\seq[n,k\in\N]{h_{n,k}}$, whose nonzero entries are
$h_{n,n+1}=1$ and $h_{n+k,n}=\gamma^{[k]}_n$ for all $n\in\N$ and $0\leq k\leq r$.
%
The $r$-orthogonal polynomials $\seq{P_n(x)}$ are the characteristic polynomials of the $(n\times n)$-matrices $\mathrm{H}_n$ formed by the first $n$ rows and columns of $\mathrm{H}$ (see \cite[\S 2.2]{WalterCoussementGaussianQuadMOP}), i.e., 
$P_n(x)=\det\left(x\,\mathrm{I}_n-\mathrm{H}_n\right)$
for any $n\geq 1$,
where $\mathrm{I}_n$ denotes the $(n\times n)$-identity matrix.


For $m\geq 2$, a polynomial sequence $\seq{P_n(x)}$ is \textit{$m$-fold symmetric} if
\begin{equation}
\label{m-fold symmetry definition}
P_n\left(\e^{\frac{2\pi i}{m}}x\right)=\e^{\frac{2n\pi i}{m}}P_n(x)
\quad\text{for all }n\in\N.
\end{equation}
Observe that, for $m=2$, \eqref{m-fold symmetry definition} corresponds to the definition of a symmetric polynomial sequence.

The study of multiple orthogonal polynomial sequences satisfying \eqref{m-fold symmetry definition} goes back to works of Douak and Maroni (see, e.g., \cite{DouakandMaroniClassiquesDeDimensionDeux}), where a $d$-symmetric sequence corresponds to what we call here a $(d+1)$-fold symmetric sequence. 
The terminology of $m$-fold symmetry in the context of multiple orthogonality was introduced in \cite{AnaWalter3FoldSym}.

The definition \eqref{m-fold symmetry definition} is equivalent to the existence of $m$ polynomial sequences $\dis\seq{P_n^{[j]}(x)}$, $0\leq j\leq m-1$, 
such that
\begin{equation}
\label{m-fold decomposition}
P_{mn+j}(x)=x^j\,P^{[j]}_n\left(x^m\right)
\quad\text{for all }n\in\N\text{ and }0\leq j\leq m-1.
\end{equation}

Moreover, for $r\in\Z^+$, $\seq{P_n(x)}$ is a $(r+1)$-fold symmetric $r$-orthogonal polynomial sequence if and only if it satisfies a three-term recurrence relation of order $r+1$ of the form
\begin{equation}
\label{recurrence relation for a r+1-fold symmetric r-OPS}
P_{n+r+1}(x)=x\,P_{n+r}(x)-\alpha_n\,P_n(x)
\quad\text{for all }n\in\N,
\end{equation}
with $\alpha_n\neq 0$ for all $n\in\N$ and initial conditions $P_j(x)=x^j$ for $0\leq j\leq r$ (see \cite[Thm.~5.1]{DouakandMaroniClassiquesDeDimensionDeux}).

The corresponding Hessenberg matrix, which we denote by $\mathrm{H}^{(r)}=\seq[n,k\in\N]{h^{(r)}_{m,n}}$, only has nonzero entries in the supradiagonal and in the $r^\text{th}$-subdiagonal, respectively equal to $1$ and to the elements of $\seq{\alpha_n}$, i.e.,
\begin{equation}
\label{Hessenberg matrix r+1-fold sym r-OP entries}
h_{n,n+1}=1
\quad\text{and}\quad
h_{n+r,n}=\alpha_n
\quad\text{for all  }n\in\N,
\quad\text{and}\quad
h_{m,n}=0
\text{ whenever }m-n\not\in\{-1,r\}.
\end{equation}

It is known that a sequence $\seq{P_n(x)}$ satisfying \eqref{recurrence relation for a r+1-fold symmetric r-OPS} with $\alpha_n>0$ for all $n\in\N$, is $r$-orthogonal with respect to a vector of positive measures $\left(\mu_1,\cdots,\mu_r\right)$ supported on the star-like set
\begin{equation}
\label{r+1-star}
\mathrm{St}^{(r+1)}=\bigcup\limits_{k=0}^{r}\left\{\left.x\,\e^{\frac{2\pi ik}{r+1}}\right|x\geq 0\right\},
\end{equation}
and uniquely determined by their densities on the positive real line.
This result was implicitly obtained in \cite{AptKalVanIseGeneticSums}, although the set $\mathrm{St}^{(r+1)}$ is not mentioned therein (see \cite[Eq.~2.15]{WalterNonsymmetric} and \cite[Thm.~1.1]{AptekarevKalyaginSaff3-termRecMOP} for explicit statements of this result).
%
Furthermore, it was shown in \cite[Thm.~2.2]{BenRomdhaneZeros-of-d-Symmetric-d-OP} that the zeros of a sequence $\seq{P_n(x)}$ satisfying \eqref{recurrence relation for a r+1-fold symmetric r-OPS} with $\alpha_n>0$ for all $n\in\N$ are all located on $\mathrm{St}^{(r+1)}$, are invariant under rotations of angle $\frac{2k\pi}{r+1}$ for $1\leq k\leq r$, and satisfy interlacing properties that we revisit in Subsection \ref{Zeros section}.

We refer to $\mathrm{St}^{(r+1)}$ as the $(r+1)$-star.
Note that the ``$2$-star" corresponds to the real line.
In Fig. \ref{3,4,5-star fig.}, we show representations of the $m$-stars for $m\in\{3,4,5\}$.
\begin{figure}[h]
\centering
\begin{tikzpicture}[scale=1]
\draw[->,color=black] (0,0) -- (2,0);
\node[above] at (1,0){$\R^+$};
\draw[->,color=black] (0,0) -- (-1.5,1.5);
\node[above,rotate=315] at (-0.75,0.75){$\e^{\frac{2\pi i}{3}}\R^+$};
\draw[->,color=black] (0,0) -- (-1.5,-1.5);
\node[above,rotate=45] at (-0.75,-0.75){$\e^{\frac{4\pi i}{3}}\R^+$};
\filldraw[black] (0,0) circle (2pt);
\node[below](0,0){$0$};
\node at (0,-2.5) {$\mathrm{St}^{(3)}$};
\end{tikzpicture}
\quad
\begin{tikzpicture}[scale=1]
\draw[->,color=black] (0,0) -- (2,0);
\node[above] at (1,0){$\R^+$};
\draw[->,color=black] (0,0) -- (0,2);
\node[right] at (0,1){$i\R^+$};
\draw[->,color=black] (0,0) -- (-2,0);
\node[above] at (-1,0){$\R^-$};
\draw[->,color=black] (0,0) -- (0,-2);
\node[right] at (0,-1){$i\R^-$};
\filldraw[black] (0,0) circle (2pt);
\node[below right] (0,0) {$0$};
\node at (0,-2.5) {$\mathrm{St}^{(4)}$};
\end{tikzpicture}
\quad
\begin{tikzpicture}[scale=1]
\draw[->,color=black] (0,0) -- (2,0);
\node[above] at (1.5,0){$\R^+$};
\draw[->,color=black] (0,0) -- (0.62,1.9);
\node[right] at (0.46,1.43){$\e^{\frac{2\pi i}{5}}\R^+$};
\draw[->,color=black] (0,0) -- (-1.62,1.18);
\node[above,rotate=324] at (-0.81,0.59){$\e^{\frac{4\pi i}{5}}\R^+$};
\draw[->,color=black] (0,0) -- (-1.62,-1.18);
\node[above,rotate=36] at (-0.81,-0.59){$\e^{\frac{-4\pi i}{5}}\R^+$};
\draw[->,color=black] (0,0) -- (0.62,-1.9);
\node[right] at (0.46,-1.43){$\e^{\frac{-2\pi i}{5}}\R^+$};
\filldraw[black] (0,0) circle (2pt);
\node[below right] (0,0) {$0$};
\node at (0,-2.5) {$\mathrm{St}^{(5)}$};
\end{tikzpicture}
\caption{$3$-star, $4$-star, and $5$-star (from left to right).}
\label{3,4,5-star fig.}
\end{figure}
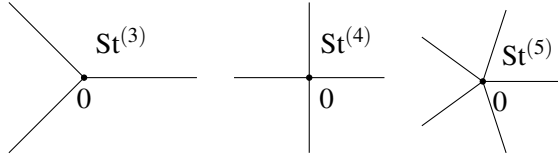

\subsection{Lattice paths and production matrices}
\label{Background lattice paths and production matrices}
We are interested in two types of lattice paths, $r$-Dyck paths and $r$-Lukasiewick paths, where $r$ is a positive integer, and on generating polynomials of sets of these lattice paths.
We follow the definitions in \cite{AlanEtAl-LPandBCF1,AlanEtAl-LPandBCF2}.

For $r\in\Z^+$, a \textit{$r$-Dyck path} is a path in the lattice $\N\times\N$, starting and ending at height $0$, only using steps $(1,1)$, called \textit{rises}, and $(1,-r)$, called \textit{$r$-falls}.
When $r=1$, the $1$-Dyck paths are simply known as Dyck paths.
More generally, a \textit{partial $r$-Dyck path} is a path in $\N\times\N$ allowed to start and end anywhere, again only using steps $(1,1)$ and $(1,-r)$.

For $r\in\Z^+\cup\{\infty\}$, a \textit{$r$-Lukasiewick path} is a path in $\N\times\N$, starting and ending at height $0$, only using steps $(1,j)$ with $-r\leq j\leq 1$.
A $1$-Lukasiewick path is called a \textit{Motzkin path} and a $\infty$-Lukasiewick path is simply called a \textit{Lukasiewick path}.
Likewise for $r$-Dyck paths, we can consider \textit{partial $r$-Lukasiewick paths}, which are paths in $\N\times\N$, only using steps $(1,j)$ with $-r\leq j\leq 1$ and allowed to start and end anywhere.
Observe that a (partial) $r$-Lukasiewick path without any steps $(1,j)$ with $1-r\leq j\leq 0$ is a (partial) $r$-Dyck path.

In Fig. \ref{2-Dyck path and 3-Lukasiewick path fig.}, we give examples of a $r$-Dyck path and a $r$-Lukasiewick path, with $r=2$ and $r=3$, respectively.
\begin{figure}[ht]
\centering
\begin{tikzpicture}[scale=0.5]
\draw[->,color=black] (0,0) -- (0,5);
\draw[->,color=black] (0,0) -- (13,0);
\filldraw[black] (0,0) circle (4pt);
\draw[-,color=black] (0,0) -- (1,1);
\filldraw[black] (1,1) circle (4pt);
\draw[-,color=black] (1,1) -- (2,2);
\filldraw[black] (2,2) circle (4pt);
\draw[-,color=black] (2,2) -- (3,3);
\filldraw[black] (3,3) circle (4pt);
\draw[-,color=black] (3,3) -- (4,4);
\filldraw[black] (4,4) circle (4pt);
\draw[-,color=black] (4,4) -- (5,2);
\filldraw[black] (5,2) circle (4pt);
\draw[-,color=black] (5,2) -- (6,3);
\filldraw[black] (6,3) circle (4pt);
\draw[-,color=black] (6,3) -- (7,1);
\filldraw[black] (7,1) circle (4pt);
\draw[-,color=black] (7,1) -- (8,2);
\filldraw[black] (8,2) circle (4pt);
\draw[-,color=black] (8,2) -- (9,0);
\filldraw[black] (9,0) circle (4pt);
\draw[-,color=black] (9,0) -- (10,1);
\filldraw[black] (10,1) circle (4pt);
\draw[-,color=black] (10,1) -- (11,2);
\filldraw[black] (11,2) circle (4pt);
\draw[-,color=black] (11,2) -- (12,0);
\filldraw[black] (12,0) circle (4pt);
\end{tikzpicture}
\quad\quad
\begin{tikzpicture}[scale=0.5]
\draw[->,color=black] (0,0) -- (0,5);
\draw[->,color=black] (0,0) -- (14,0);
\filldraw[black] (0,0) circle (4pt);
\draw[-,color=black] (0,0) -- (1,1);
\filldraw[black] (1,1) circle (4pt);
\draw[-,color=black] (1,1) -- (2,2);
\filldraw[black] (2,2) circle (4pt);
\draw[-,color=black] (2,2) -- (3,3);
\filldraw[black] (3,3) circle (4pt);
\draw[-,color=black] (3,3) -- (4,3);
\filldraw[black] (4,3) circle (4pt);
\draw[-,color=black] (4,3) -- (5,1);
\filldraw[black] (5,1) circle (4pt);
\draw[-,color=black] (5,1) -- (6,2);
\filldraw[black] (6,2) circle (4pt);
\draw[-,color=black] (6,2) -- (7,3);
\filldraw[black] (7,3) circle (4pt);
\draw[-,color=black] (7,3) -- (8,4);
\filldraw[black] (8,4) circle (4pt);
\draw[-,color=black] (8,4) -- (9,1);
\filldraw[black] (9,1) circle (4pt);
\draw[-,color=black] (9,1) -- (10,0);
\filldraw[black] (10,0) circle (4pt);
\draw[-,color=black] (10,0) -- (11,1);
\filldraw[black] (11,1) circle (4pt);
\draw[-,color=black] (11,1) -- (12,2);
\filldraw[black] (12,2) circle (4pt);
\draw[-,color=black] (12,2) -- (13,0);
\filldraw[black] (13,0) circle (4pt);
\end{tikzpicture}
\caption{A $2$-Dyck path of length $12$ (left) and a $3$-Lukasiewick path of length $13$ (right).}
\label{2-Dyck path and 3-Lukasiewick path fig.}
\end{figure}
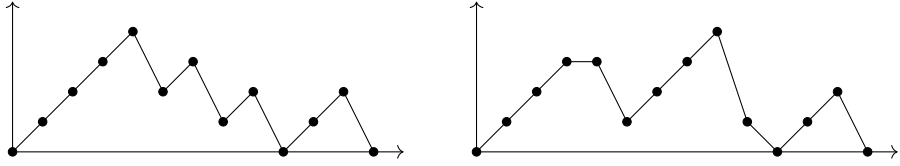

Throughout this paper and as usual, the weight of a lattice path is the product of the weights of its steps and the generating polynomial of a set of lattice paths is the sum of the weights of the paths in the set.

Observe that the length of a $r$-Dyck path is always a multiple of $r+1$.
For any $r\in\Z^+$ and a sequence $\seq[i\in\N]{\alpha_i}$, the \textit{$r$-Stieltjes-Rogers polynomials} $\StieltjesRogersPoly[r]{n}{\mathbf{\alpha}}$, with $n\in\N$, are the generating polynomials of the sets of $r$-Dyck paths of length $(r+1)n$, where each rise has weight $1$ and each $r$-fall to height $i$ has weight $\alpha_i$. 
Analogously, for $r\in\Z^+\cup\{\infty\}$ and sequences $\left(\gamma_i^{[j]}\right)_{i\in\N}$, for $0\leq j\leq r$, the \textit{$r$-Jacobi-Rogers polynomials} $\JacobiRogersPoly[r]{n}{\gamma}$, with $n\in\N$, are the generating polynomials of the sets of $r$-Lukasiewicz paths from $(0,0)$ to $(n,0)$, where each step $(1,1)$ gets weight $1$ and each step $(1,-j)$, with $0\leq j\leq r$, to height $i$ gets weight $\gamma_i^{[j]}$. 

The $r$-Stieltjes-Rogers and $r$-Jacobi-Rogers polynomials were introduced in \cite{AlanEtAl-LPandBCF1} as extensions of the Stieltjes-Rogers and Jacobi-Rogers polynomials, which correspond to the case $r=1$, introduced by Rogers in \cite{Rogers1907} and also studied by Flajolet in \cite{FlajoletContinuedFractions}.
We are interested in further generalisations of these polynomials.

Observe that every vertex $(x,y)$ of a partial $r$-Dyck path starting at $(0,0)$, and in particular its final vertex, satisfies $x=(r+1)m+y$ for some $m\in\N$, where $x$ is the total number of steps and $m$ is the number of $r$-falls. 
For a sequence $\dis\left(\alpha_i\right)_{i\in\N}$, the \textit{generalised $r$-Stieltjes-Rogers polynomials of type $j$}, $\generalisedStieltjesRogersPolyTypeJ{n}{k}{j}{\mathbf{\alpha}}$, with $n,k,j\in\N$, are the generating polynomials of the sets of partial $r$-Dyck paths from $(0,0)$ to $\big((r+1)n+j,(r+1)k+j\big)$, where each rise has weight $1$ and each $r$-fall to height $i$ has weight $\alpha_i$.
These polynomials were introduced in \cite[\S~A.3]{AlanEtAl-LPandBCF3}.
When $k=0$, the generalised $r$-Stieltjes-Rogers polynomials of type $j$ reduce to the \textit{modified $r$-Stieltjes-Rogers polynomials of type $j$}, $\modifiedStieltjesRogersPoly[r]{n}{j}{\mathbf{\alpha}}$, the generating polynomials of the set of partial $r$-Dyck paths from $(0,0)$ to $((r+1)n+j,j)$.
The \textit{generalised $r$-Jacobi-Rogers polynomials}, $\generalisedJacobiRogersPoly[r]{n}{k}{\gamma}$, with $n,k\in\N$, are the generating polynomials of partial $r$-Lukasiewicz paths from $(0,0)$ to $(n,k)$, where each step $(1,1)$ has weight $1$ and each step $(1,-j)$, with $0\leq j\leq r$, to height $i$ has weight $\gamma_i^{[j]}$.

The method of production matrices was originally introduced in \cite{ProductionMatrices2005}.
Let $\mathrm{H}=\left(h_{i,j}\right)_{i,j\in\N}$ be an infinite matrix with entries in a commutative ring $R$.
If all the powers of $\mathrm{H}$ are well-defined (for example, if all rows or columns of $\mathrm{H}$ have finite nonzero entries), we can define an infinite matrix $A=\left(a_{n,k}\right)_{n,k\in\N}$ such that $a_{n,k}=\left(\mathrm{H}^n\right)_{0,k}$.
We call $\mathrm{H}$ the \textit{production matrix} and $A$ the \textit{output matrix}.

Writing out the matrix multiplications explicitly, we get
\begin{equation}
\label{output matrix entries formula}
a_{n,k}=\sum_{i_1,\cdots,i_{n-1}\in\N}\;h_{0,i_1}h_{i_1,i_2}\cdots h_{i_{n-2},i_{n-1}}h_{i_{n-1},k}.
\end{equation}
Therefore, $a_{n,k}$ is the sum of the weights of all $n$-step walks in $\N$ from $i_0=0$ to $i_n=k$, where the weight of a walk is the product of the weights of its steps and a step from $i$ to $j$ has weight $h_{i,j}$.

Based on \cite[Prop.~8.2]{AlanEtAl-LPandBCF1}, the production matrices of $\mathrm{J}^{(r)}=\seq[n,k\in\N]{\generalisedJacobiRogersPoly[r]{n}{k}{\gamma}}$ and $\mathrm{S}^{(r;0)}=\left(\generalisedStieltjesRogersPolyTypeJ[r]{n}{k}{0}{\mathbf{\alpha}}\right)_{n,k\in\N}$, the matrices of generalised $r$-Jacobi-Rogers polynomials and generalised $r$-Stieltjes-Rogers polynomials of type $0$, respectively, are $(r+2)$-banded unit-lower-Hessenberg matrices:
\begin{itemize}
\item 
the production matrix of $\mathrm{J}^{(r)}=\seq[n,k\in\N]{j^{(r)}_{n,k}}$ has nonzero entries
$j^{(r)}_{n,n+1}=1$ and $j^{(r)}_{n+k,n}=\gamma^{[k]}_n$ for all $n\in\N$ and $0\leq k\leq r$,
so it is the recurrence matrix of the polynomial sequence $\seq{P_n(x)}$ satisfying \eqref{recurrence relation for a r-OPS},
\item 
the production matrix of $\mathrm{S}^{(r;0)}$ is $\mathrm{H}^{(r;0)}=\mathrm{L}_1\cdots\mathrm{L}_r\,\mathrm{U}$,
where $\mathrm{L}_1,\cdots,\mathrm{L}_r$ and $\mathrm{U}$ are the infinite bidiagonal matrices in \eqref{bidiagonal matrices U and L_k}.
\end{itemize}

\section{Bidiagonal matrix factorisations for banded Hessenberg matrices}
\label{Bidiagonal factorisations for Hessenberg matrices}
In this section, we present the two main novel results of this paper: 
we show that the banded Hessenberg matrices admitting the bidiagonal matrix factorisation given by \eqref{products of lower bidiagonal matrices with an upper bidiagonal matrix in the middle}-\eqref{bidiagonal matrices U and L_k} are the recurrence matrices for the components of the decomposition of $(r+1)$-fold symmetric $r$-orthogonal polynomials (Theorem \ref{decomposition of a (r+1)-fold sym r-OPS main th.}) 
and that the same Hessenberg matrices are the production matrices of the generating polynomials of partial $r$-Dyck paths (Theorem \ref{production matrix of generalised m-S.R. poly of type j th.}).
Moreover, we give explicit expressions for the nontrivial entries of those Hessenberg matrices, which are the recurrence coefficients of the corresponding $r$-orthogonal polynomial sequences, and present an alternative proof of a total-positivity result for generalised $r$-Stieltjes-Rogers polynomials, firstly obtained in \cite{AlanEtAl-LPandBCF3}.

\subsection{Hessenberg matrices associated with symmetric multiple orthogonal polynomials}
\label{Characterisation of the components}
The components of the decomposition of a $(r+1)$-fold symmetric $r$-orthogonal polynomial sequence $\seq{P_n(x)}$ are also $r$-orthogonal (see \cite[Thm.~5.2]{DouakandMaroniClassiquesDeDimensionDeux}).
In Theorem \ref{decomposition of a (r+1)-fold sym r-OPS main th.}, we show that the recurrence matrices of those components are the banded Hessenberg matrices determined by \eqref{products of lower bidiagonal matrices with an upper bidiagonal matrix in the middle}-\eqref{bidiagonal matrices U and L_k}. 
\begin{theorem}
\label{decomposition of a (r+1)-fold sym r-OPS main th.}
For $r\in\Z^+$ and a sequence $\seq[n\in\N]{\alpha_n}$ of nonzero elements in a field $\mathbb{K}$, let:
\begin{itemize}
\item 
$\seq{P_n(x)}$ be the $(r+1)$-fold symmetric $r$-orthogonal polynomial sequence satisfying the recurrence relation 
\begin{equation}
\label{recurrence relation symmetric r-OP}
P_{n+r+1}(x)=x\,P_{n+r}(x)-\alpha_n\,P_n(x)
\quad\text{for all }n\in\N,
\end{equation}
with initial conditions $P_j(x)=x^j$ for all $0\leq j\leq r$, 
\item 
$\seq{P_n^{[j]}(x)}$, $0\leq j\leq r$, be the components of the $(r+1)$-fold decomposition of $\seq{P_n(x)}$, i.e.,
\begin{equation}
\label{r+1 fold decomposition of a symmetric r-OP}
P_{(r+1)n+j}(x)=x^j\,P_n^{[j]}\left(x^{r+1}\right)
\quad\text{for all }n\in\N\text{ and }0\leq j\leq r.
\end{equation} 
\end{itemize}

Then, for any $0\leq j\leq r$, the recurrence matrix for $\seq{P_n^{[j]}(x)}$ is the infinite $(r+2)$-banded unit-lower-Hessenberg matrix $\mathrm{H}^{(r;j)}$ admitting the factorisation \eqref{products of lower bidiagonal matrices with an upper bidiagonal matrix in the middle}, i.e., $\seq{P_n^{[j]}(x)}$ satisfies the recurrence relation
\begin{equation}
\label{recurrence relation r+1 fold decomposition components matrix form}
x\left[P_n^{[j]}(x)\right]_{n\in\N}
=\mathrm{H}^{(r;j)}\left[P_n^{[j]}(x)\right]_{n\in\N}
=\mathrm{L}_{j+1}\cdots\mathrm{L}_r\,\mathrm{U}\,\mathrm{L}_1\cdots\mathrm{L}_j
\left[P_n^{[j]}(x)\right]_{n\in\N},
\end{equation}
involving the bidiagonal matrices $\mathrm{L}_1,\cdots,\mathrm{L}_r$ and $\mathrm{U}$ in \eqref{bidiagonal matrices U and L_k}.
\end{theorem}

\begin{proof}
Let $n\in\N$ and $0\leq j\leq r$.
Combining the definition \eqref{r+1 fold decomposition of a symmetric r-OP} of $\seq{P_n^{[j]}(x)}$ with the recurrence relation \eqref{recurrence relation symmetric r-OP} satisfied by $\seq{P_n(x)}$, we find that
\begin{equation}
P_{n+1}^{[j]}(x)
=x^{-\frac{j}{r+1}}\,P_{(r+1)(n+1)+j}\left(x^{\frac{1}{r+1}}\right)
=x^{\frac{1-j}{r+1}}\,P_{(r+1)(n+1)+(j-1)}\left(x^{\frac{1}{r+1}}\right)
-\alpha_{(r+1)n+j}\,x^{-\frac{j}{r+1}}\,P_{(r+1)n+j}\left(x^{\frac{1}{r+1}}\right).
\end{equation}
Using again the definition \eqref{r+1 fold decomposition of a symmetric r-OP} of $\seq{P_n^{[j]}(x)}$, we have
\begin{equation}
x^{-\frac{j}{r+1}}\,P_{(r+1)n+j}\left(x^{\frac{1}{r+1}}\right)=P_n^{[j]}(x)
\end{equation}
and
\begin{equation}
x^{\frac{1-j}{r+1}}\,P_{(r+1)(n+1)+(j-1)}\left(x^{\frac{1}{r+1}}\right)=
\begin{cases}
\hfil x\,P_n^{[r]}(x) & \text{if }j=0, \\
   P_{n+1}^{[j-1]}(x) & \text{if }1\leq j\leq r.
\end{cases}
\end{equation}
Therefore, we find that, for any $n\in\N$,
\begin{equation}
\label{relations between components of sym r-OPS}
P_{n+1}^{[0]}(x)=x\,P_n^{[r]}(x)-\alpha_{(r+1)n}\,P_n^{[0]}(x)
\quad\text{and}\quad
P_{n+1}^{[j]}(x)=P_{n+1}^{[j-1]}(x)-\alpha_{(r+1)n+j}\,P_n^{[j]}(x)
\text{ for }1\leq j\leq r.
\end{equation}
Using the bidiagonal matrices $\mathrm{L}_1,\cdots,\mathrm{L}_r$ and $\mathrm{U}$ in \eqref{bidiagonal matrices U and L_k}, we can rewrite \eqref{relations between components of sym r-OPS} as
\begin{equation}
\label{relation between consecutive components}
x\left[P_n^{[r]}(x)\right]_{n\in\N}=\mathrm{U}\,\left[P_n^{[0]}(x)\right]_{n\in\N}
\quad\text{and}\quad
\left[P_n^{[j-1]}(x)\right]_{n\in\N}=\mathrm{L}_j\,\left[P_n^{[j]}(x)\right]_{n\in\N}
\text{ for }1\leq j\leq r.
\end{equation}
	
Combining both equations in \eqref{relation between consecutive components}, we obtain 
\begin{equation}
\label{relation between non-consecutive components 1}
x\left[P_n^{[r]}(x)\right]_{n\in\N}
=\mathrm{U}\,\mathrm{L}_1\cdots\mathrm{L}_j\left[P_n^{[j]}(x)\right]_{n\in\N}
\quad\text{for any }1\leq j\leq r.
\end{equation}
In particular, we have
\begin{equation}
x\left[P_n^{[r]}(x)\right]_{n\in\N}
=\mathrm{U}\,\mathrm{L}_1\cdots\mathrm{L}_r\left[P_n^{[r]}(x)\right]_{n\in\N}
=\mathrm{H}^{(r;r)}\left[P_n^{[r]}(x)\right]_{n\in\N}.
\end{equation}
Therefore, \eqref{recurrence relation r+1 fold decomposition components matrix form} holds for $j=r$.
	
Now let $0\leq j\leq r-1$.
Using successively the second equation in \eqref{relation between consecutive components}, we find that
\begin{equation}
\label{relation between non-consecutive components 2}
x\left[P_n^{[j]}(x)\right]_{n\in\N}
=\mathrm{L}_{j+1}\cdots\mathrm{L}_r\,x\left[P_n^{[r]}(x)\right]_{n\in\N}.
\end{equation}
Combining the latter with \eqref{relation between non-consecutive components 1}, we find that
\begin{equation}
x\left[P_n^{[j]}(x)\right]_{n\in\N}
=\mathrm{L}_{j+1}\cdots\mathrm{L}_r\,\mathrm{U}\,\mathrm{L}_1\cdots\mathrm{L}_j \left[P_n^{[j]}(x)\right]_{n\in\N}
=\mathrm{H}^{(r;j)}\left[P_n^{[j]}(x)\right]_{n\in\N}.
\end{equation}
Hence, \eqref{recurrence relation r+1 fold decomposition components matrix form} also holds for $0\leq j\leq r-1$.
\end{proof}

As a consequence of Theorem \ref{decomposition of a (r+1)-fold sym r-OPS main th.}, the components $\seq{P_n^{[j]}(x)}$ satisfy $(r+1)$-order recurrence relations whose coefficients corresponding to the nontrivial entries of $\mathrm{H}^{(r;j)}=\mathrm{L}_{j+1}\cdots\mathrm{L}_r\,\mathrm{U}\,\mathrm{L}_1\cdots\mathrm{L}_j$ for all $0\leq j\leq r$.
Therefore, we obtain the following result.
\begin{corollary}
\label{decomposition of a (r+1)-fold sym r-OPS rec. coef.}
For $r\in\Z^+$ and a sequence $\seq[n\in\N]{\alpha_n}$ of nonzero elements in a field $\mathbb{K}$, let $\seq{P_n^{[j]}(x)}$, for $0\leq j\leq r$, be the same $r$-orthogonal polynomial sequences as in Theorem \ref{decomposition of a (r+1)-fold sym r-OPS main th.}. 
Then, 
\begin{equation}
\label{recurrence relation r+1 fold decomposition components}
P_{n+1}^{[j]}(x)=x\,P_n^{[j]}(x)-\sum_{k=0}^{\min(r,n)}\gamma_{n-k}^{\,[k;j]}\,P_{n-k}^{[j]}(x)
\quad\text{for any }n\in\N\text{ and }0\leq j\leq r,
\end{equation}
with initial conditions $P_0^{[j]}(x)=1$ and coefficients
\begin{equation}
\label{recurrence coefficients r+1 fold decomposition components}
\gamma_n^{\,[k;j]}
=\sum_{r\geq t_0>\cdots>t_k\geq 0}\;\prod_{i=0}^{k}\alpha_{(r+1)(n+i)+t_i+j-r}
\quad\text{for any }n\in\N\text{ and }0\leq j,k\leq r,
\quad\text{with }\alpha_m=0\text{ if }m<0.
\end{equation}
\end{corollary}

Explicit expressions for the coefficients \eqref{recurrence coefficients r+1 fold decomposition components} of the recurrence relation \eqref{recurrence relation r+1 fold decomposition components} were already known for $r=1,2$.
For $r=1$, see \cite[Thm.~9.1]{ChiharaBook}; for $r=2$, see \cite[\S~5.1]{DouakandMaroniClassiquesDeDimensionDeux} and \cite[Lemma~2.1]{AnaWalter3FoldSym}.

To obtain the formula \eqref{recurrence coefficients r+1 fold decomposition components} for the recurrence coefficients in \eqref{recurrence relation r+1 fold decomposition components}, we use the following lemma.
\begin{lemma}
\label{entries of a product of bidiagonal matrices permutated lemma}
For $r\in\Z^+$, $0\leq j\leq r$, and a sequence $\seq[k\in\N]{\alpha_k}$ in a commutative ring $R$, 
let $\mathrm{H}^{(r;j)}=\seq[m,n\in\N]{h^{(r;j)}_{m,n}}$ be defined by \eqref{products of lower bidiagonal matrices with an upper bidiagonal matrix in the middle}-\eqref{bidiagonal matrices U and L_k}. 
Then, for any $m,n\in\N$, $h^{(r;j)}_{m,n}$ is the generating polynomial of the partial $r$-Dyck paths from $(0,(r+1)m+j)$ to $(r+1,(r+1)n+j)$.
Therefore, $\mathrm{H}^{(r;j)}$ is a ${(r+2)}$-banded unit-lower-Hessenberg matrix with nontrivial entries
\begin{equation}
\label{formula for the entries of a product of bidiagonal matrices permutated}
h^{(r;j)}_{n+k,n}
=\sum_{r\geq t_0>\cdots>t_k\geq 0}\;\prod_{i=0}^{k}\alpha_{(r+1)(n+i)+t_i+j-r}
\quad\text{for all }n\in\N\text{ and }0\leq k\leq r,
\quad\text{with }\alpha_m=0\text{ if }m<0.
\end{equation}
\end{lemma}

When $j=0$, Lemma \ref{entries of a product of bidiagonal matrices permutated lemma} reduces to \cite[Prop.~3.5]{HypergeometricMOP+BCF}.
Here, we give a different proof, which we believe to be more insightful, because we give a combinatorial interpretation to the entries of $\mathrm{H}^{(r;j)}$ and use it to derive \eqref{formula for the entries of a product of bidiagonal matrices permutated}, instead of proving \eqref{formula for the entries of a product of bidiagonal matrices permutated} by induction.

\begin{proof}
By definition of $\mathrm{H}^{(r;j)}=\mathrm{L}_{j+1}\cdots\mathrm{L}_r\,\mathrm{U}\,\mathrm{L}_1\cdots\mathrm{L}_j$,
\begin{equation}
\label{product of bidiagonal matrices proof formula 1*}
h^{(r;j)}_{m,n}
=\sum_{\left(m_1,\cdots,m_r\right)\in\N^r}
\prod_{i=0}^{r-j-1}\ell^{(j+i+1)}_{m_i,m_{i+1}}
\,u_{m_{r-j},m_{r-j+1}}\,
\prod_{i=r-j+1}^{r}\ell^{(i+j-r)}_{m_i,m_{i+1}}
\quad\text{with }m_0=m\text{ and }m_{r+1}=n,
\end{equation}
where $\mathrm{U}=\seq[i,i'\in\N]{u_{i,i'}}$ and 
$\mathrm{L}_k=\seq[i,i'\in\N]{\ell^{(k)}_{i,i'}}$ for $1\leq k\leq r$.

Due to the bidiagonality of the matrices $\mathrm{L}_k$ and $\mathrm{U}$, the summands above are equal to $0$ unless
\begin{equation}
\label{product of bidiagonal matrices proof formula 2*}
m_{r-j+1}-m_{r-j}\in\{0,1\}
\quad\text{and}\quad
m_{i+1}-m_i\in\{-1,0\}
\text{ for all }i\in\{1,\cdots,r\}\backslash\{r-j\}.
\end{equation}
In particular, this implies that $h^{(r;j)}_{m,n}=0$ when $m-n\not\in\{-1,0,\cdots,r\}$.

By definition of the matrices $\mathrm{L}_1,\cdots,\mathrm{L}_r$,
\begin{equation}
\ell^{(k)}_{i,i'}=
\begin{cases}
\alpha_{(r+1)i'+k} & \text{if }i'=i-1, \\
		   \hfil 1 & \text{if }i'=i, \\
		   \hfil 0 & \text{otherwise},
\end{cases}
\quad\text{for any }1\leq k\leq r\text{ and }i,i'\in\N.
\end{equation}
Hence, $\ell^{(k)}_{i,i'}$ is the weight of the step from $\big(n,(r+1)i+k-1\big)$ to $\big(n+1,(r+1)i'+k\big)$ for any $n\in\N$ if that step is allowed in a partial $r$-Dyck path and $\ell^{(k)}_{i,i'}=0$ if that step is not allowed.

Similarly,
\begin{equation}
u_{i,i'}=
\begin{cases}
\alpha_{(r+1)i'} & \text{if }i'=i, \\
		 \hfil 1 & \text{if }i'=i+1, \\
		 \hfil 0 & \text{otherwise},
\end{cases}
\end{equation}
so $u_{i,i'}$ is the weight of the step from $\big(n,(r+1)i+r\big)$ to $\big(n+1,(r+1)i'\big)$, for any $n\in\N$, if that step is allowed in a partial $r$-Dyck path and $u_{i,i'}=0$ if that step is not allowed.

Therefore, the summand in \eqref{product of bidiagonal matrices proof formula 1*} corresponding to each $\left(m_1,\cdots,m_r\right)\in\N^r$ is equal to the weight of the partial $r$-Dyck path from $(0,(r+1)m+j)$ to $(r+1,(r+1)n+j)$, with steps $v_0\to v_1\to\cdots\to v_{r+1}$ such that
\begin{equation}
v_i=
\begin{cases}
	\hfil\left(i,(r+1)m_i+j+i\right) & \text{if }1\leq i\leq r-j, \\
\hfil\left(i,(r+1)(m_i-1)+j+i\right) & \text{if }r-j+1\leq i\leq r+1,
\end{cases}
\quad\text{with }m_0=m\text{ and }m_{r+1}=n,
\end{equation}
if this partial $r$-Dyck path exists, and it is equal to $0$ if this partial $r$-Dyck path does not exist.
Moreover, all partial $r$-Dyck paths from $(0,(r+1)m+j)$ to $(r+1,(r+1)n+j)$ are of this form, with $\left(m_1,\cdots,m_r\right)$ satisfying \eqref{product of bidiagonal matrices proof formula 2*}, because the only steps allowed are $(1,1)$ and $(1,-r)$.
As a result, we find that $h^{(r;j)}_{m,n}$ is the generating polynomial of the partial $r$-Dyck paths from $(0,(r+1)m+j)$ to $(r+1,(r+1)n+j)$.

Any partial $r$-Dyck path from $(0,(r+1)(n+k)+j)$ to $(r+1,(r+1)n+j)$ has $r+1$ steps: $k+1$ $r$-falls and $r-k$ rises.
In particular, there are no $r$-Dyck paths from $(0,(r+1)(n+k)+j)$ to $(r+1,(r+1)n+j)$ unless $-1\leq k\leq r$ and there exists only one partial $r$-Dyck path from $(0,(r+1)n+j)$ to $(r+1,(r+1)(n+1)+j)$, formed by $r+1$ rises, which all have weight $1$.
Hence, $h^{(r;j)}_{n+k,n}=0$ whenever $k<-1$ or $k>r$ and $h^{(r;j)}_{n,n+1}=1$ for any $n\in\N$. 
Therefore, $\mathrm{H}^{(r;j)}$ is a ${(r+2)}$-banded unit-lower-Hessenberg matrix.

Now we fix $0\leq k\leq r$ and compute $h^{(r;j)}_{n+k,n}$. 
Let $\mathcal{P}$ be a partial $r$-Dyck path from $(0,(r+1)(n+k)+j)$ to $(r+1,(r+1)n+j)$.
The weight of $\mathcal{P}$ is equal to the product of its $k+1$ $r$-falls.
Let 
\begin{equation}
V_{\mathcal{P}}=\left\{\left.v_i=(i,y(i))\in\N^2\right|0\leq i\leq r+1\right\}
\end{equation}
be the set of vertices of $\mathcal{P}$ and let $r\geq t_0>\cdots>t_k\geq 0$ such that the $r$-falls correspond to the steps $v_{t_i}\to v_{t_i+1}$, ordered from the right to the left.
The weight of each $r$-fall $v_{t_i}\to v_{t_i+1}$ is $\alpha_{y\left(t_i+1\right)}$.
Observe that $\mathcal{P}$ has $r-t_i$ steps to the right of the $r$-fall $v_{t_i}\to v_{t_i+1}$, and $i$ of those steps are $r$-falls. 
Hence, because the final height of $\mathcal{P}$ is $(r+1)n+j$, we have
\begin{equation}
(r+1)n+j=y\left(t_i+1\right)-ri+r-t_i-i
\Leftrightarrow
y\left(t_i+1\right)=(r+1)(n+1)+j-r+t_i.
\end{equation}
Therefore, the weight of $\mathcal{P}$ is $\dis\prod_{i=0}^{k}\alpha_{(r+1)(n+i)+t_i+j-r}$.
In addition, there is a bijection between the possible choices of $r\geq t_0>\cdots>t_k\geq 0$ and the partial $r$-Dyck paths from $(0,(r+1)(n+k)+j)$ to $(r+1,(r+1)n+j)$, because each of those paths is uniquely determined by its $r$-falls.
As a result, we obtain \eqref{formula for the entries of a product of bidiagonal matrices permutated}.
\end{proof}

\subsection{Production matrices of generating polynomials of lattice paths}
\label{Results on lattice paths and production matrices}
In this section, we show that the production matrices of the generalised $r$-Stieltjes-Rogers polynomials of type $j$, with $0\leq j\leq r$, are the banded Hessenberg matrices given by \eqref{products of lower bidiagonal matrices with an upper bidiagonal matrix in the middle}-\eqref{bidiagonal matrices U and L_k} and then we use that result to prove the total positivity of the matrices of generalised $r$-Stieltjes-Rogers polynomials of type $j$ as a consequence of the total positivity of their production matrices.
The results presented in this section, Theorem \ref{production matrix of generalised m-S.R. poly of type j th.} and 
Corollaries \ref{total positivity of the Hessenberg matrices} 
and \ref{total positivity of the generalised r-S.-R. poly of type j},
were independently obtained in \cite[\S~A.3]{AlanEtAl-LPandBCF3}, where they appeared slightly after the first version of this text was made available.

\begin{theorem}
\label{production matrix of generalised m-S.R. poly of type j th.}
For $r\in\Z^+$, $0\leq j\leq r$, and a sequence $\seq[n\in\N]{\alpha_n}$ in a commutative ring $R$, 
the production matrix of the matrix $\mathrm{S}^{(r;j)}=\seq[n,k\in\N]{\generalisedStieltjesRogersPolyTypeJ[r]{n}{k}{j}{\boldsymbol{\alpha}}}$ of generalised $r$-Stieltjes-Rogers polynomials of type $j$ is
\begin{equation}
\label{products of lower bidiagonal matrices with an upper bidiagonal matrix in the middle*}
\mathrm{H}^{(r;j)}=\mathrm{L}_{j+1}\cdots\mathrm{L}_r\,\mathrm{U}\,\mathrm{L}_1\cdots\mathrm{L}_j,
\end{equation} 
where $\mathrm{L}_1,\cdots,\mathrm{L}_r$ and $\mathrm{U}$ are the infinite bidiagonal matrices in \eqref{bidiagonal matrices U and L_k}.
%
\end{theorem}
For $j=0$, the decomposition \eqref{products of lower bidiagonal matrices with an upper bidiagonal matrix in the middle*} was found in \cite[Prop.~8.2]{AlanEtAl-LPandBCF1}.

\begin{proof}
Let $\mathrm{A}^{(r;j)}=\left(a_{n,k}^{(r;j)}\right)_{n,k\in\N}$ be the output matrix of $\mathrm{H}^{(r;j)}$.
We want to prove that 
\begin{equation}
\label{entries of the output matrix A as gen. r-S.R. poly}
a_{n,k}^{(r;j)}=\generalisedStieltjesRogersPolyTypeJ[r]{n}{k}{j}{\boldsymbol{\alpha}}
\quad\text{for all }n,k\in\N\text{ and }0\leq j\leq r.
\end{equation}

Recalling \eqref{output matrix entries formula}, $a_{n,k}^{(r;j)}$ is the sum of the weights of all $n$-step walks in $\N$ from $i_0=0$ to $i_n=k$, where the weight of a walk is the product of the weights of its steps and a step from $i$ to $j$ has weight $h_{i,j}^{(r;j)}$.

Replacing $i_m\in\N$ by $\big((r+1)m+j,(r+1)i_m+j\big)\in\N^2$ for all $0\leq m\leq n$, the latter is equivalent to saying that $a_{n,k}^{(r;j)}$ is the generating polynomial of the paths in $\N^2$ with vertex set
\begin{equation}
V=\left\{\left.v_m=\big((r+1)m+j,(r+1)i_m+j\big)\in\N^2\,\right|\,0\leq m\leq n\right\}
\quad\text{with }i_0=0\text{ and }i_n=k,
\end{equation}
and steps $v_m\to v_{m+1}$ with weights $h_{i_m,i_{m+1}}^{(r;j)}$, for $0\leq m\leq n-1$. 

Using Lemma \ref{entries of a product of bidiagonal matrices permutated lemma}, $h_{i_m,i_{m+1}}^{(r;j)}$ is the generating polynomial of the partial $r$-Dyck paths from $(0,(r+1)i_m+j)$ to $(r+1,(r+1)i_{m+1}+j)$.
Moreover, the weight of a partial $r$-Dyck path is invariable with horizontal shifts, because it only depends on its height after each $r$-fall.
So, $h_{i_m,i_{m+1}}^{(r;j)}$ is also the generating polynomial of the partial $r$-Dyck paths from $v_m$ to $v_{m+1}$.
Therefore, $a_{n,k}^{(r;j)}$ is the generating polynomial of the partial $r$-Dyck paths from $v_0=(j,j)$ to $v_n=\big((r+1)n+j,(r+1)k+j\big)$.

In addition, the first $j$ steps (in fact, the first $r$ steps) of any partial $r$-Dyck path starting at $(0,0)$ are all rises, so any partial $r$-Dyck path starting at $(0,0)$ passes at $(j,j)$ and has the same weight as the path starting at $(j,j)$ obtained by removing its first $j$ steps, which are all rises, so have weight $1$.
As a result, $a_{n,k}^{(r;j)}$ is the generating polynomial of the partial $r$-Dyck paths from $(0,0)$ to $\big((r+1)n+j,(r+1)k+j\big)$, which means that \eqref{entries of the output matrix A as gen. r-S.R. poly} holds,
concluding our proof.
\end{proof}

Following the terminology used in \cite{AlanEtAl-LPandBCF1,AlanEtAl-LPandBCF2,AlanEtAl-LPandBCF3,PinkusTotallyPositiveMatrices}, we say that a matrix with real entries is \textit{totally positive} if all its minors are nonnegative (different terminology is used on some other references on total positivity, e.g. \cite{GantmacherKreinOscillationMatrices,FallatJohnsonTotallyNonnegativeMatrices}).
The definition of a totally positive matrix can be naturally extended for matrices with entries in any partially ordered commutative ring (see \cite[\S 9.1]{AlanEtAl-LPandBCF1} for more details).
A bidiagonal matrix is totally positive if and only if all its entries are nonnegative, because all its nonzero minors are products of entries (see \cite[Lemma~9.1]{AlanEtAl-LPandBCF1}).
Moreover, using the Cauchy-Binet formula (see \cite[Prop.~1.4]{PinkusTotallyPositiveMatrices}), the product of matrices preserves total positivity.
Therefore, we have the following result as a direct consequence of Theorem \ref{production matrix of generalised m-S.R. poly of type j th.}.
\begin{corollary}
\label{total positivity of the Hessenberg matrices}
For $r\in\Z^+$ and $0\leq j\leq r$, the matrix $\mathrm{H}^{(r;j)}$ defined by \eqref{products of lower bidiagonal matrices with an upper bidiagonal matrix in the middle*} is totally positive in the polynomial ring $\Z[\alpha]$ equipped with the coefficient-wise partial order, i.e., all minors of $\mathrm{H}^{(r;j)}$ are elements of $\Z[\alpha]$ whose coefficients are all nonnegative.
\end{corollary}

Furthermore, the output matrix of a totally positive matrix is also totally positive (see \cite[Thm.~9.4]{AlanEtAl-LPandBCF1}).
As a result, Theorem \ref{production matrix of generalised m-S.R. poly of type j th.} implies that the matrix of generalised $r$-Stieltjes-Rogers polynomials of type $j$, $\mathrm{S}^{(r;j)}$, is coefficient-wise totally positive for any $r\in\Z^+$ and $0\leq j\leq r$.
Moreover, we can reduce the generalised $r$-Stieltjes-Rogers polynomials of type higher than $j$ to the ones of type $0\leq j\leq r$ via the relation 
\begin{equation}
\label{relation between generalised m-S.-R. poly of type j for any j in N and j<=m}
\generalisedStieltjesRogersPolyTypeJ[r]{n}{k}{(r+1)q+j}{\mathbf{\alpha}}
=\generalisedStieltjesRogersPolyTypeJ[r]{n+q}{k+q}{j}{\mathbf{\alpha}} 
\quad\text{for any }r\in\Z^+,\,n,k,q\in\N,\text{ and }0\leq j\leq r,
\end{equation}
which is a direct consequence of the definition of the generalised $r$-Stieltjes-Rogers polynomials.
As a consequence, $\mathrm{S}^{(r;(r+1)q+j)}$ is a submatrix of $\mathrm{S}^{(r;j)}$ obtained by removing its first $q$ rows and columns.
Therefore, we obtain the following result.
\begin{corollary}
\label{total positivity of the generalised r-S.-R. poly of type j}
For any $r\in\Z^+$ and $j\in\N$, the matrix 
$\mathrm{S}^{(r;j)}=\seq[n,k\in\N]{\generalisedStieltjesRogersPolyTypeJ[r]{n}{k}{j}{\boldsymbol{\alpha}}}$, is totally positive in the polynomial ring $\Z[\alpha]$ equipped with the coefficient-wise partial order, i.e., all minors of $\mathrm{S}^{(r;j)}$ are elements of $\Z[\alpha]$ whose coefficients are all nonnegative.
\end{corollary}

\section{On $(r+1)$-fold symmetric $r$-orthogonal polynomials}
\label{Symmetric MOP}
We start this section by using Theorem \ref{decomposition of a (r+1)-fold sym r-OPS main th.} and the spectral properties of oscillation matrices to obtain new, simpler proofs of known properties of the zeros of $(r+1)$-fold symmetric $r$-orthogonal polynomial sequences with positive recurrence coefficients, which we collect in Proposition \ref{zero location (r+1)-fold sym r-OPS}.

Next, we use Theorem \ref{production matrix of generalised m-S.R. poly of type j th.} to find combinatorial interpretations for the moments of the dual sequences of $(r+1)$-fold symmetric $r$-orthogonal polynomial sequences and of the components of their decomposition, as described in Proposition \ref{dual sequences of a (r+1)-fold sym r-OPS and its components}.

These combinatorial interpretations are used to find orthogonality measures on the $(r+1)$-star for $(r+1)$-fold symmetric $r$-orthogonal polynomial sequences and positive orthogonality measures on the positive real line for the components of their decomposition (see Theorems \ref{multiple orthogonality measures on the (r+1)-star th.*} and \ref{orthogonality functionals and measures for the components}, respectively). 
The existence of these orthogonality measures is well known since it was already found in \cite{AptKalVanIseGeneticSums}, but we present here a simpler proof of their existence and give a novel combinatorial interpretation for their nonzero moments.

Finally, we present, in Theorem \ref{particular case th.}, an explicit example of $(r+1)$-fold symmetric $r$-orthogonal polynomials, which we show that are Appell sequences, can be represented as terminating hypergeometric series, and are multiple orthogonal with respect to measures that can be expressed via Meijer G-functions.

\subsection{Location and interlacing properties of the zeros}
\label{Zeros section}
\begin{proposition}(cf. \cite[Thm.~2.2]{BenRomdhaneZeros-of-d-Symmetric-d-OP})
\label{zero location (r+1)-fold sym r-OPS}
For $r\in\Z^+$ and a sequence of positive real numbers $\seq[n\in\N]{\alpha_n}$, 
let $\seq{P_n(x)}$ be the $(r+1)$-fold symmetric $r$-orthogonal polynomial sequence satisfying the recurrence relation \eqref{recurrence relation symmetric r-OP} 
and let $\seq{P_n^{[j]}(x)}$, with $0\leq j\leq r$, be the components of the decomposition \eqref{r+1 fold decomposition of a symmetric r-OP} of $\seq{P_n(x)}$. 
Then, the following statements hold for any $n\in\Z^+$ and $0\leq j\leq r$:
\begin{enumerate}[label=(\alph*)]
\item 
$P_n^{[j]}$ has $n$ simple, real, and positive zeros, which we denote, in increasing order, by $x_1^{[n;j]},\cdots,x_n^{[n;j]}$.

\item
The zeros of $P_n^{[j]}$ and $P_{n+1}^{[j]}$ interlace, i.e.:
\begin{equation}
\label{interlacing of the zeros of the components of the r+1-fold decomposition}
x_1^{[n+1;j]}<x_1^{[n;j]}<x_2^{[n+1;j]}<x_2^{[n;j]}<\cdots<x_n^{[n+1;j]}<x_n^{[n;j]}<x_{n+1}^{[n+1;j]}.
\end{equation}

\item
$P_{(r+1)n+j}$ has a zero of multiplicity $j$ at the origin if $j\geq 1$ and has $(r+1)n$ simple zeros that correspond to the $(r+1)^{\text{th}}$-roots of the zeros of $P_n^{[j]}$. 
Hence, $x\neq 0$ is a zero of $P_{(r+1)n+j}$ if and only if
\begin{equation}
x=\sqrt[r+1]{x_m^{[n;j]}}\exp\left(\frac{2\pi i k}{r+1}\right)
\quad\text{for some }0\leq k\leq r\text{ and }1\leq m\leq n,
\end{equation} 
where $\sqrt[r+1]{x_m^{[n;j]}}$ is the real positive $(r+1)^{\text{th}}$-root of $x_m^{[n;j]}$.
Therefore, the zeros of $P_{(r+1)n+j}$ are all located on the star-like set $\mathrm{St}^{(r+1)}$ given in \eqref{r+1-star} and the $n$ zeros on each ray of $\mathrm{St}^{(r+1)}$ are copies of the $n$ positive real zeros of $P_{(r+1)n+j}$ rotated by angles of $\frac{2k\pi}{r+1}$ for each $0\leq k\leq r$. 
\end{enumerate}
\end{proposition}

We show that the proof of those properties becomes considerably simpler than in \cite[Thm.~2.2]{BenRomdhaneZeros-of-d-Symmetric-d-OP}, using the recurrence relation \eqref{recurrence relation r+1 fold decomposition components matrix form} from Theorem \ref{decomposition of a (r+1)-fold sym r-OPS main th.} and the spectral properties of oscillation matrices. 

Oscillation matrices are a class of finite square matrices intermediary between totally positive and strictly totally positive matrices. \index{oscillation matrix}
A finite square matrix $A$ with real entries is an \textit{oscillation matrix} if $A$ is totally positive and some power of $A$ is strictly totally positive.
We are interested in oscillation matrices because of their useful spectral properties: the eigenvalues of an oscillation matrix are all simple, real, and positive, and interlace with the eigenvalues of the submatrices obtained by removing either its first or last column and row.
This result is known as the Gantmacher-Krein theorem and can be found in \cite[Ths.~II-6~\&~II-14]{GantmacherKreinOscillationMatrices}.

\begin{lemma}
\label{H^(r;j) are oscillation matrices}
For $r\in\Z^+$ and a sequence $\seq[k\in\N]{\alpha_k}$ of positive real numbers, let $\mathrm{H}^{(r;j)}$, $0\leq j\leq r$, be the infinite banded Hessenberg matrices defined by \eqref{products of lower bidiagonal matrices with an upper bidiagonal matrix in the middle}-\eqref{bidiagonal matrices U and L_k}.
The finite leading principal submatrices $\mathrm{H}_n^{(r;j)}$, with $n\geq 1$, of $\mathrm{H}^{(r;j)}$, formed by its first $n$ rows and columns, are oscillation matrices.
\end{lemma}

\begin{proof}
We know from Lemma \ref{total positivity of the Hessenberg matrices} that the matrices $\mathrm{H}^{(r;j)}$, $0\leq j\leq r$, are totally positive when $\alpha_k\geq 0$ for all $k\in\N$. 
Hence, its submatrices $\mathrm{H}_n^{(r;j)}$ are also totally positive for all $0\leq j\leq r$ and $n\geq 1$.

Furthermore, based on \cite[Thm.~5.2]{PinkusTotallyPositiveMatrices}, a real $(n\times n)$-matrix is an oscillation matrix if and only if it is totally positive, nonsingular, and all the entries lying in its subdiagonal and its supradiagonal are positive.
Therefore, when $\alpha_k>0$ for all $k\in\N$, $\mathrm{H}_n^{(r;j)}$ is an oscillation matrix for any $0\leq j\leq r$ and $n\geq 1$.
\end{proof}

\begin{proof}[Proof of Proposition \ref{zero location (r+1)-fold sym r-OPS}]
Recall that the eigenvalues of an oscillation matrix are all simple, real, and positive, and interlace with the eigenvalues of the submatrices obtained by removing either its first or last column and row. 
Hence, parts (a) and (b) hold as a consequence of Lemma \ref{H^(r;j) are oscillation matrices}, because $P_n^{[j]}(x)$ is the characteristic polynomial of $\mathrm{H}_n^{(r;j)}$ for any $n\geq 1$ and $0\leq j\leq r$, thus the zeros of $P_n^{[j]}(x)$ are the eigenvalues of $\mathrm{H}_n^{(r;j)}$.
Part (c) is obtained by combining (a) with the $(r+1)$-fold decomposition \eqref{r+1 fold decomposition of a symmetric r-OP}. 
\end{proof}


\subsection{Orthogonality measures on a star-like set and on the positive real line}
\label{Connection with lattice paths}
In the following result, we determine the moments of the dual sequences of a $(r+1)$-fold symmetric $r$-orthogonal polynomial sequence and of the components of its decomposition.
\begin{proposition}
\label{dual sequences of a (r+1)-fold sym r-OPS and its components}
For $r\in\Z^+$ and a sequence $\seq[n\in\N]{\alpha_n}$ of nonzero elements of a field $\mathbb{K}$, let 
$\seq[k\in\N]{u_k}$ be the dual sequence of the $(r+1)$-fold symmetric $r$-orthogonal polynomial sequence $\seq{P_n(x)}$ satisfying the recurrence relation \eqref{recurrence relation symmetric r-OP} 
and let $\seq[k\in\N]{u_k^{[j]}}$, for $0\leq j\leq r$, be the dual sequences of the components $\seq{P_n^{[j]}(x)}$ of the decomposition \eqref{r+1 fold decomposition of a symmetric r-OP} of $\seq{P_n(x)}$.
Then:
\begin{enumerate}[label=(\alph*)]
\item 
$\seq[k\in\N]{u_k}$ is $(r+1)$-fold symmetric, that is, 
$\Functional{u_k}{x^n}=0$ for any $k,n\in\N$ such that $n\not\equiv_{r+1}k$,
		
\item 
$\Functional{u_{(r+1)k+j}}{x^{(r+1)n+j}}=\Functional{u_k^{[j]}}{x^n}
=\generalisedStieltjesRogersPolyTypeJ[r]{n}{k}{j}{\boldsymbol{\alpha}}$ 
for any $n,k\in\N$ and $0\leq j\leq r$.
\end{enumerate}
\end{proposition}

Part (a) and the first equality of part (b) of the theorem above were already known for $r\in\{1,2\}$.
For $r=1$, see \cite[Thm.~8.1]{ChiharaBook};
for $r=2$, see \cite{DouakandMaroniClassiquesDeDimensionDeux,MaroniOrthogonalite} and \cite[Lemma~2.1~\&~Prop.~2.1]{AnaWalter3FoldSym}.

\begin{remark}
\label{genetic sums and modified S.-R.-poly}
When $k=0$, the moments $\generalisedStieltjesRogersPolyTypeJ[r]{n}{k}{j}{\boldsymbol{\alpha}}$ appearing in part (b) of Proposition \ref{dual sequences of a (r+1)-fold sym r-OPS and its components}  are equal to the ``genetic sums" defined in \cite[Thm.~1]{AptKalVanIseGeneticSums}.
In fact, setting $p=r$ and $a_{i+1}=\alpha_i$ for all $i\in\N$, \cite[Eq.~9]{AptKalVanIseGeneticSums} becomes
\begin{equation}
\label{genetic sums}
S_{0,j+1}=1
\quad\text{and}\quad
S_{n,j+1}=\sum_{i_1=0}^{j}\,\sum_{i_2=0}^{i_1+r}\cdots\sum_{i_n=0}^{i_{n-1}+r}\alpha_{i_1}\cdots\alpha_{i_n}
\text{ if }n\geq 1
\quad\text{for }0\leq j\leq r-1.
\end{equation}
Observe that each summand $\alpha_{i_1}\cdots\alpha_{i_n}$ is equal to the weight of a partial $r$-Dyck path of length $(r+1)n+j$ starting at height $0$, with $n$ $r$-falls to heights $i_1,\cdots,i_n$, from right to left, thus ending at height $j$.
Therefore, $S_{n,j+1}$, with $n\in\N$ and $0\leq j\leq r-1$, is the generating polynomial of the partial $r$-Dyck paths from $(0,0)$ to $\left((r+1)n+j,j\right)$, i.e., $S_{n,j+1}=\modifiedStieltjesRogersPoly[r]{n}{j}{\boldsymbol{\alpha}}$, giving a combinatorial interpretation to the ``genetic sums" in \cite{AptKalVanIseGeneticSums}.
\end{remark}

We introduce here matrices representing sequences of linear functionals and polynomials:
\begin{itemize}
\item 
the \textit{moment matrix} of a sequence of linear functionals $\seq[k\in\N]{u_k}$ is $A=\seq[n,k\in\N]{a_{n,k}}$ such that $a_{n,k}=\Functional{u_k}{x^n}$,

\item 
the \textit{coefficient matrix} of a polynomial sequence $\seq{P_n(x)}$ is $B=\seq[n,k\in\N]{b_{n,k}}$ such that $P_n(x)=\sum\limits_{k=0}^{n}b_{n,k}\,x^k$.
\end{itemize}
Using their representing matrices, $\seq[k\in\N]{u_k}$ is the dual sequence of $\seq{P_n(x)}$ if and only if the moment matrix of $\seq[k\in\N]{u_k}$ is the inverse of the coefficient matrix of $\seq{P_n(x)}$ 
(see \cite[Prop.~3.1]{AlanSokal.MOP-ProdMat-BCF}).

To prove Proposition \ref{dual sequences of a (r+1)-fold sym r-OPS and its components}, we use the following result linking production matrices and polynomial sequences.
\begin{proposition}
\label{PolySeqAndHessMatrixProp}
\cite[Prop.~3.2]{AlanSokal.MOP-ProdMat-BCF}
Given a monic polynomial sequence $\seq{P_n(x)}$, there exists a unique unit-lower-Hessenberg matrix $\mathrm{H}=\seq[n,k\in\N]{h_{n,k}}$ such that
\begin{equation}
\label{recurrence relation generic PS}
x\left[P_n(x)\right]_{n\in\N}=\mathrm{H}\left[P_n(x)\right]_{n\in\N}.
\end{equation}
Conversely, given any unit-lower-Hessenberg matrix $\mathrm{H}=\seq[n,k\in\N]{h_{n,k}}$, the recurrence relation \eqref{recurrence relation generic PS} with the initial condition $P_0(x)=1$ determines uniquely the polynomial sequence $\seq{P_n(x)}$ formed by the characteristic polynomials of the leading $(n\times n)$-submatrices of $\mathrm{H}$ for all $n\geq 1$.

Furthermore, the coefficient matrix of the sequence $\seq{P_n(x)}$ is the inverse of the output matrix of $\mathrm{H}$. 
Therefore, $\mathrm{H}$ is the production matrix of the moment matrix of the dual sequence of $\seq{P_n(x)}$.
\end{proposition}

\begin{proof}[Proof of Proposition \ref{dual sequences of a (r+1)-fold sym r-OPS and its components}]
Recalling Theorem \ref{production matrix of generalised m-S.R. poly of type j th.}, $\mathrm{H}^{(r;j)}=\mathrm{L}_{j+1}\cdots\mathrm{L}_r\,\mathrm{U}\,\mathrm{L}_1\cdots\mathrm{L}_j$ is the production matrix of $\seq[n,k\in\N]{\generalisedStieltjesRogersPolyTypeJ[r]{n}{k}{j}{\boldsymbol{\alpha}}}$.
Therefore, using Proposition \ref{PolySeqAndHessMatrixProp}, the recurrence relation \eqref{recurrence relation r+1 fold decomposition components matrix form} implies that 
\begin{equation}
\label{moments of the dual seq of the (r+1)-fold decomposition components}
\Functional{u_k^{[j]}}{x^n}=\generalisedStieltjesRogersPolyTypeJ[r]{n}{k}{j}{\boldsymbol{\alpha}}
\quad\text{for any }n,k\in\N \text{ and } 0\leq j\leq r.
\end{equation}

Using again Proposition \ref{PolySeqAndHessMatrixProp}, the $(r+2)$-banded Hessenberg matrix $\mathrm{H}^{(r)}$ whose entries are given in \eqref{Hessenberg matrix r+1-fold sym r-OP entries} is the production matrix of the moment matrix of $\seq[k\in\N]{u_k}$, i.e., the output matrix of $\mathrm{H}^{(r)}$ is $A=\seq[k,n\in\N]{\Functional{u_k}{x^n}}$.
As a result, based on \cite[Prop.~8.2]{AlanEtAl-LPandBCF1}, $\Functional{u_k}{x^n}$ is the generating polynomial of partial $r$-Lukasiewicz paths from $(0,0)$ to $(n,k)$, where each rise has weight $1$, each level step or $\ell$-fall with $1\leq\ell\leq r-1$ has weight $0$, and each $r$-fall to height $i$ has weight $\alpha_i$.
This is equivalent to say that $\Functional{u_k}{x^n}$ is the generating polynomial of the partial $r$-Dyck paths from $(0,0)$ to $(n,k)$, because all partial $r$-Lukasiewicz paths that include steps that are neither a rise or a $r$-fall have weight $0$.
Therefore:
\begin{itemize}
\item 
$\Functional{u_k}{x^n}=0$ for all $k,n\in\N$ such that $n\not\equiv_{r+1}k$, because each vertex $(n,k)$ of a partial $r$-Dyck path starting at $(0,0)$ satisfies $n\equiv_{r+1}k$,

\item 
$\Functional{u_{(r+1)k+j}}{x^{(r+1)n+j}}=\generalisedStieltjesRogersPolyTypeJ[r]{n}{k}{j}{\boldsymbol{\alpha}}$, by definition of the generalised $m$-Stieltjes-Rogers polynomials. \qedhere
\end{itemize}
\end{proof}

As a consequence of Proposition \ref{dual sequences of a (r+1)-fold sym r-OPS and its components}, $\seq{P_n(x)}$ is $r$-orthogonal with respect to $\left(u_0,\cdots,u_{r-1}\right)$ such that
\begin{equation}
\label{orthogonality functionals dual seq. r+1 fold sym r-OPS}
\Functional{u_j}{x^{(r+1)n+k}}=
\begin{cases}
\modifiedStieltjesRogersPoly[r]{n}{j}{\boldsymbol{\alpha}} & \text{if }k=j, \\
												   \hfil 0 & \text{if }k\neq j,
\end{cases}
\quad\text{for all }0\leq j\leq r-1,\;0\leq k\leq r,\text{ and }n\in\N.
\end{equation}
In the following result, we show that $\seq{P_n(x)}$ is $r$-orthogonal with respect to a system of measures on the star-like set $\mathrm{St}^{(r+1)}$ whose moments correspond to the moments of $u_0,\cdots,u_{r-1}$.
\begin{theorem}
\label{multiple orthogonality measures on the (r+1)-star th.*}
For $r\in\Z^+$ and a sequence of positive real numbers $\seq[n\in\N]{\alpha_n}$, let $\seq{P_n(x)}$ be the $(r+1)$-fold symmetric $r$-orthogonal polynomial sequence satisfying the recurrence relation \eqref{recurrence relation symmetric r-OP}.
Then, $\seq{P_n(x)}$ is $r$-orthogonal with respect to a system of measures $\left(\hat{\mu}_1,\cdots,\hat{\mu}_r\right)$ on the star-like set $\mathrm{St}^{(r+1)}$ in \eqref{r+1-star}, with densities
\begin{equation}
\label{density of the orthogonality measures on the star}
\mathrm{d}\hat{\mu}_j(z)=\frac{1}{r+1}\,z^{1-j}\,\mathrm{d}\mu_j\left(z^{r+1}\right)
\quad\text{for all }1\leq j\leq r\text{ and }z\in\mathrm{St}^{(r+1)},
\end{equation}
where $\mu_1,\cdots,\mu_r$ are positive measures on $\R^+$ with moments
\begin{equation}
\label{measures on R+ whose moments are modified r-S.R. poly}
\int_{0}^{\infty}x^n\mathrm{d}\mu_j(x)=\modifiedStieltjesRogersPoly[r]{n}{j-1}{\alpha}
\quad\text{for all }n\in\N\text{ and }1\leq j\leq r.
\end{equation}
Furthermore, the system of measures $\left(\hat{\mu}_1,\cdots,\hat{\mu}_r\right)$ is $(r+1)$-fold symmetric, that is
\begin{equation}
\label{zero moments of the orthogonality measures on the star}
\int_{\mathrm{St}^{(r+1)}}z^m\mathrm{d}\hat{\mu}_j(z)=0
\quad\text{when }m\not\equiv_{r+1}j-1,
\end{equation}
and the nonzero moments are given by
\begin{equation}
\label{nonzero moments of the orthogonality measures on the star}
\int_{\mathrm{St}^{(r+1)}}z^{(r+1)n+j-1}\mathrm{d}\hat{\mu}_j(z)
=\int_{0}^{\infty}x^n\mathrm{d}\mu_j(x)
=\modifiedStieltjesRogersPoly[r]{n}{j-1}{\alpha}
\quad\text{for all }n\in\N\text{ and }1\leq j\leq r.
\end{equation}
\end{theorem}

As mentioned in the introduction, this result was implicitly obtained in \cite{AptKalVanIseGeneticSums}, while explicit statements were given in \cite[Eq.~2.15]{WalterNonsymmetric} and \cite[Thm.~1.1]{AptekarevKalyaginSaff3-termRecMOP}.
The rotational invariance up to a monomial factor of the measures $\hat{\mu}_2,\cdots,\hat{\mu}_r$ in \eqref{density of the orthogonality measures on the star} was given in \cite[Lemma~8.1]{DelvauxLopezHighOrder3TermRec}, assuming that the sequence $\seq[n\in\N]{\alpha_n}$ is bounded.

However, as far as I am aware, the first explicit proof of this result is given here.
Moreover, the proof appears as a natural consequence of Proposition \ref{dual sequences of a (r+1)-fold sym r-OPS and its components}, which gives us the combinatorial interpretation for the nonzero moments of the orthogonality measures in Theorem \ref{multiple orthogonality measures on the (r+1)-star th.*}.

\begin{proof}
Based on \cite[Cor.~4.6]{HypergeometricMOP+BCF}, there exist positive measures $\mu_1,\cdots,\mu_r$ on $\R^+$ satisfying \eqref{measures on R+ whose moments are modified r-S.R. poly}.
As a consequence of Remark \ref{genetic sums and modified S.-R.-poly}, this statement is equivalent to \cite[Thm.~7]{AptKalVanIseGeneticSums}.

We show that \eqref{measures on R+ whose moments are modified r-S.R. poly} implies that the measures $\hat{\mu}_1,\cdots,\hat{\mu}_r$ on $\mathrm{St}^{(r+1)}$ defined by \eqref{density of the orthogonality measures on the star} satisfy \eqref{zero moments of the orthogonality measures on the star}-\eqref{nonzero moments of the orthogonality measures on the star}, which means that $\hat{\mu}_1,\cdots,\hat{\mu}_r$ have the same moments as the first $r$ elements of the dual sequence of $\seq{P_n(x)}$. 
As a result, $\seq{P_n(x)}$ is $r$-orthogonal with respect to $\left(\hat{\mu}_1,\cdots,\hat{\mu}_r\right)$.
 
Recall that $\dis\mathrm{St}^{(r+1)}=\bigcup_{k=0}^{r}\Gamma_k$, where $\dis\Gamma_k=\left\{\omega^k\,x\,|\,x\geq 0\right\}$, with $\omega=\e^\frac{2\pi i}{r+1}$, for each $0\leq k\leq r$.

Using \eqref{density of the orthogonality measures on the star}, we have, for any $1\leq j\leq r$ and $m\in\N$,
\begin{equation}
\int_{\mathrm{St}^{(r+1)}}z^m\mathrm{d}\hat{\mu}_j(z)
=\frac{1}{r+1}\sum_{k=0}^{r}\int_{\Gamma_k}z^{m-j-1}\,\mathrm{d}\mu_j\left(z^{r+1}\right).
\end{equation}
Making the change of variable $z=\omega^k\,x^{\frac{1}{r+1}}\in\Gamma_k$, which implies that $x=z^{r+1}\in\R^+$, we derive
\begin{equation}
\label{integral on the star after change of variable}
\int_{\mathrm{St}^{(r+1)}}z^m\mathrm{d}\hat{\mu}_j(z)
=\frac{1}{r+1}\sum_{k=0}^{r}\omega^{(m-j-1)k}\int_{0}^{\infty}x^{\frac{m-j-1}{r+1}}\,\mathrm{d}\mu_{j+1}(x).
\end{equation}
If $m\not\equiv_{r+1}j-1$, then $\omega^{m-j-1}\neq 1$ and $\omega^{(m-j-1)(r+1)}=1$, so \eqref{zero moments of the orthogonality measures on the star} holds because
\begin{equation}
\sum_{k=0}^{r}\omega^{(m-j-1)k}=\frac{\omega^{(m-j-1)(r+1)}-1}{\omega^{m-j-1}-1}=0.
\end{equation}
Otherwise, we can write $m=(r+1)n+j-1$ with $n\in\N$. 
Then, $\omega^{m-j-1}=\omega^{(r+1)n}=1$ and $\dfrac{m-j-1}{r+1}=n$.
Hence, \eqref{integral on the star after change of variable} gives the first equality in \eqref{nonzero moments of the orthogonality measures on the star}.
\end{proof}

As a consequence of Proposition \ref{dual sequences of a (r+1)-fold sym r-OPS and its components}, the sequences $\seq{P_n^{[j]}(x)}$, with $0\leq j\leq r$, are $r$-orthogonal with respect to $\left(u_0^{[j]},\cdots,u_{r-1}^{[j]}\right)$ such that $\Functional{u_k^{[j]}}{x^n}=\generalisedStieltjesRogersPolyTypeJ[r]{n}{k}{j}{\boldsymbol{\alpha}}$ for any $n\in\N$ and $0\leq k\leq r-1$.
We are interested in writing the multiple orthogonality of these sequences with respect to positive measures on $\R^+,$ but $u_k^{[j]}$ cannot be induced by a positive measure on $\R^+$ for any $k\geq 1$, because $\Functional{u_k^{[j]}}{1}=0$.
Hence, we find other multiple orthogonality functionals for $\seq{P_n^{[j]}(x)}$, induced by positive measures on $\R^+$ when the coefficients $\alpha_n$ are all positive.
\begin{theorem}
\label{orthogonality functionals and measures for the components}
For $r\in\Z^+$ and a sequence $\seq[n\in\N]{\alpha_n}$ of nonzero elements of a field $\mathbb{K}$, 
let $\seq{P_n(x)}$ be the $(r+1)$-fold symmetric $r$-orthogonal polynomial sequence satisfying the recurrence relation \eqref{recurrence relation symmetric r-OP}, 
$\seq{P_n^{[j]}(x)}$, with $0\leq j\leq r$, be the components of the decomposition \eqref{r+1 fold decomposition of a symmetric r-OP} of $\seq{P_n(x)}$, 
and $v_1,\cdots,v_r$ be the linear functionals with moments
\begin{equation}
\label{modified r-S.R. poly as moments of the orthogonality functionals}
\Functional{v_i}{x^n}
=\modifiedStieltjesRogersPoly[r]{n}{i-1}{\boldsymbol{\alpha}}
\quad\text{for all }n\in\N\text{ and }1\leq i\leq r.
\end{equation}
Then, for any $0\leq j\leq r$, $\seq{P_n^{[j]}(x)}$ is $r$-orthogonal with respect to the system of linear functionals $\left(v_{j+1},\cdots,v_r,x\,v_1,\cdots,x\,v_j\right)$.

If $\alpha_n>0$ for all $n\in\N$, let $\mu_1,\cdots,\mu_r$ be positive measures on $\R^+$ satisfying \eqref{measures on R+ whose moments are modified r-S.R. poly}.
Then, for any $0\leq j\leq r$, $\seq{P_n^{[j]}(x)}$ is $r$-orthogonal with respect to the system of measures $\left(\mu_1^{[j]},\cdots,\mu_r^{[j]}\right)$ on $\R^+$ such that
\begin{equation}
\label{orthogonality measures components (r+1)-fold sym r-OPS}
\mathrm{d}\mu_i^{[j]}(x)=\mathrm{d}\mu_{j+i}(x)
\quad\text{for}\quad 1\leq i\leq r-j
\quad\text{and}\quad
\mathrm{d}\mu_{r-j+i}^{[j]}(x)=x\,\mathrm{d}\mu_i(x)
\quad\text{for}\quad 1\leq i\leq j.
\end{equation}
\end{theorem}

\begin{proof}
Let $\seq[k\in\N]{u_k}$ be the dual sequence of $\seq{P_n(x)}$.
Recalling Proposition \ref{dual sequences of a (r+1)-fold sym r-OPS and its components}, \eqref{modified r-S.R. poly as moments of the orthogonality functionals} is equivalent to
\begin{equation}
\label{orthogonality functionals of the components and the dual seq of P_n}
\Functional{v_i}{x^n}
=\Functional{u_{i-1}}{x^{(r+1)n+i-1}}
\quad\text{for all }n\in\N\text{ and }1\leq i\leq r.
\end{equation}
where $\seq[k\in\N]{u_k}$ is the dual sequence of $\seq{P_n(x)}$.
Moreover, because $\seq{P_n(x)}$ is $r$-orthogonal with respect to $\left(u_0,\cdots,u_{r-1}\right)$, we have
\begin{equation}
\label{r-orthogonality conditions dual seq. of P_n}
\Functional{u_{i-1}}{x^k\,P_n(x)}
=\begin{cases}
N_n\neq 0 &\text{ if } n=rk+i-1, \\
  \hfil 0 &\text{ if } n\geq rk+i,
\end{cases}
\quad\text{for all }1\leq i\leq r.
\end{equation}

Fix $0\leq j\leq r$.
For $1\leq i\leq r-j$, we can use \eqref{orthogonality functionals of the components and the dual seq of P_n} and the $(r+1)$-fold decomposition \eqref{r+1 fold decomposition of a symmetric r-OP} of $\seq{P_n(x)}$ to find that
\begin{equation}
\Functional{v_{j+i}}{x^k\,P_n^{[j]}(x)}
=\Functional{u_{j+i-1}}{x^{(r+1)k+j+i-1}\,P_n^{[j]}\left(x^{r+1}\right)}
=\Functional{u_{j+i-1}}{x^{(r+1)k+i-1}\,P_{(r+1)n+j}(x)}.
\end{equation}
If $n=rk+i-1$, then $(r+1)n+j=r\left((r+1)k+i-1\right)+j+i-1$, so, using \eqref{r-orthogonality conditions dual seq. of P_n}, we have
\begin{equation}
\label{r-orthogonality conditions components 1}
\Functional{v_{j+i}}{x^k\,P_n^{[j]}(x)}
=\Functional{u_{j+i-1}}{x^{(r+1)k+i-1}\,P_{(r+1)n+j}(x)}
=\begin{cases}
N_n\neq 0 & \text{if } n=rk+i-1, \\
  \hfil 0 & \text{if } n\geq rk+i.
\end{cases}
\end{equation}

For $1\leq i\leq j$, we can again use \eqref{orthogonality functionals of the components and the dual seq of P_n} and the $(r+1)$-fold decomposition \eqref{r+1 fold decomposition of a symmetric r-OP} of $\seq{P_n(x)}$ to find that
\begin{equation}
\Functional{x\,v_i}{x^k\,P_n^{[j]}(x)}
=\Functional{u_{i-1}}{x^{(r+1)(k+1)+i-1}\,P_n^{[j]}\left(x^{r+1}\right)}
=\Functional{u_{i-1}}{x^{(r+1)k+r-j+i}\,P_{(r+1)n+j}(x)}.
\end{equation}
If $n=rk+r-j+i-1$, then $(r+1)n+j=r\big((r+1)k+r-j+i\big)+i-1$, so, using \eqref{r-orthogonality conditions dual seq. of P_n}, we have
\begin{equation}
\label{r-orthogonality conditions components 2}
\Functional{x\,v_i}{x^k\,P_n^{[j]}(x)}
=\Functional{u_{i-1}}{x^{(r+1)k+r-i-j}\,P_{(r+1)n+j}(x)}
=\begin{cases}
N_n\neq 0 & \text{if } n=rk+r-j+i-1, \\
  \hfil 0 & \text{if } n\geq rk+r-j+i.
\end{cases}
\end{equation}

Combining \eqref{r-orthogonality conditions components 1} and \eqref{r-orthogonality conditions components 2}, we conclude that $\seq{P_n^{[j]}(x)}$ is $r$-orthogonal with respect to the system of linear functionals $\left(v_{j+1},\cdots,v_r,x\,v_1,\cdots,x\,v_j\right)$.

If $\alpha_n>0$ for all $n\in\N$, there exist positive measures $\mu_1,\cdots,\mu_r$ on $\R^+$ satisfying  \eqref{measures on R+ whose moments are modified r-S.R. poly}, i.e., with the same moments as $v_1,\cdots,v_r$.
Therefore, $\seq{P_n^{[j]}(x)}$ is $r$-orthogonal with respect to the system $\left(\mu_1^{[j]},\cdots,\mu_r^{[j]}\right)$ determined by \eqref{orthogonality measures components (r+1)-fold sym r-OPS}.
\end{proof}


\subsection{An explicit example involving hypergeometric Appell polynomials}
\label{explicit example}
We finish this paper giving an explicit example of a $(r+1)$-fold symmetric $r$-orthogonal polynomial sequence $\seq{P_n(x)}$ satisfying the recurrence relation \eqref{recurrence relation symmetric r-OP} for a specific sequence of coefficients $\seq{\alpha_n}$.
These polynomials can be expressed via terminating hypergeometric series and form Appell sequences, i.e.,
\begin{equation}
\label{Appell property}
P_{n+1}^{\,\prime}(x)=(n+1)P_n(x)
\quad\text{for all }n\in\N.
\end{equation}
The Hermite polynomials form the only Appell orthogonal polynomial sequence, up to a linear transformation (see \cite[Thm.~IV]{ShohatClassicalOPandAppell}).
An investigation of the Appell $r$-orthogonal polynomials is presented in \cite{Douakd-OPAppell}.
The polynomial sequences described in Theorem \ref{particular case th.} are the only Appell $(r+1)$-fold symmetric $r$-orthogonal polynomial sequences, up to a scaling change of variable.

The components of the $(r+1)$-fold decomposition of the polynomials in Theorem \ref{particular case th.} are particular cases of the multiple orthogonal polynomials with respect to Meijer G-functions $G^{\,r,0}_{0,r}$ on the positive real line introduced in \cite{KuijlaarsZhang14} to investigate singular values of products of Ginibre random matrices.
Therefore, we can show that the corresponding $(r+1)$-fold symmetric polynomials are multiple orthogonal with respect to measures on the star-like set $\mathrm{St}^{(r+1)}$, whose densities can be expressed via the same Meijer G-functions.

Before stating Theorem \ref{particular case th.}, we recall the definitions of the hypergeometric series and the Meijer G-function. 
For $r,s\in\N$, the (generalised) hypergeometric series (see \cite{AndrewsAskeyRoySpecialFunctions,LukeSpecialFunctionsVolI,DLMF}) is defined by
\begin{equation}
\label{hypergeometric series}
\Hypergeometric[z]{s}{r}{b_1,\cdots,b_s}{a_1,\cdots,a_r}
=\sum_{n=0}^{\infty}\frac{\pochhammer{b_1}\cdots\pochhammer{b_s}} {\pochhammer{a_1}\cdots\pochhammer{a_r}}\,\frac{z^n}{n!},
\end{equation}
where the Pochhammer symbol $\pochhammer{c}$ is defined by $\pochhammer[0]{c}=1$ and $\pochhammer[n]{c}:=c(c+1)\cdots(c+n-1)$ for $n\geq 1$.
The hypergeometric series appearing here are terminating, thus they converge and define polynomials.

For $r\in\Z^+$, the Meijer G-function $G^{\,r,0}_{0,r}$ (see \cite{LukeSpecialFunctionsVolI, DLMF} for more details and for information on the general Meijer G-function $G^{\,m,n}_{s,r}$) is defined by the Mellin-Barnes type integral
\begin{equation}
\label{Meijer G-function definition}
\MeijerG{r,0}{0,r}{-}{a_1,\cdots,a_r}
=\frac{1}{2\pi i}\int_{c-i\infty}^{c+i\infty}\Gamma\left(a_1+u\right)\cdots\Gamma\left(a_r+u\right)x^{-u}\mathrm{d}u,
\quad c>-\min\limits_{1\leq i\leq r}\{\Real(a_i)\}.
\end{equation}
As usual, $\Gamma(z)$ denotes Euler's gamma function.
A fundamental property of the Meijer G-function $G^{\,r,0}_{0,r}$ is that its Mellin transform, when it exists, is a product of gamma function values
(see \cite[Eq.~2.24.2.1]{PrudnikovEtAlVol3}):
\begin{equation}
\label{Mellin transform of a Meijer G-function}
\int\limits_{0}^{\infty}\MeijerG{r,0}{0,r}{-}{a_1,\cdots,a_r}x^{z-1}\,\mathrm{d}x
=\Gamma\left(a_1+z\right)\cdots\Gamma\left(a_r+z\right),
\quad\Real(z)>-\min\limits_{1\leq i\leq r}\{\Real\left(a_i\right)\}.
\end{equation}
Another useful property of the Meijer G-function $G^{\,r,0}_{0,r}$, which can be easily deduced from \eqref{Meijer G-function definition} and can be found in \cite[Eq.~5.3.8]{BatemanProjectHTFunctionsVol1}, is
\begin{equation}
\label{multiplication by x of a Meijer G-function}
x\,\MeijerG{r,0}{0,r}{-}{a_1,\cdots,a_r}=\MeijerG{r,0}{0,r}{-}{a_1+1,\cdots,a_r+1}.
\end{equation}

\begin{theorem}
\label{particular case th.}
For $r\in\Z^+$, let $\seq{P_n(x)}$ be the $(r+1)$-fold symmetric polynomial sequence satisfying the recurrence relation 
\begin{equation}
\label{recurrence relation symmetric r-OP hypergeometric case}
P_{n+r+1}(x)=x\,P_{n+r}(x)-\frac{\pochhammer[r]{n+1}}{(r+1)^r}\,P_n(x)
\quad\text{for all }n\in\N,
\end{equation}
with initial conditions $P_j(x)=x^j$ for all $0\leq j\leq r$.
Then:
\begin{enumerate}[label=(\alph*),leftmargin=*]
\item 
$\seq{P_n(x)}$ is $r$-orthogonal with respect to the system of measures $\left(\hat{\mu}_1,\cdots,\hat{\mu}_r\right)$ on the star-like set $\mathrm{St}^{(r+1)}$ with densities
\begin{equation}
\label{density of the orthogonality measures on the star Appell case}
\mathrm{d}\hat{\mu}_j(z)
=\MeijerG[z^{r+1}]{r,0}{0,r}{-}{a_1^{[j-1]},\cdots,a_r^{[j-1]}}\frac{\mathrm{d}z}{z^{r+j}}
\quad\text{for all }1\leq j\leq r\text{ and }z\in\mathrm{St}^{(r+1)},
\end{equation}
with
\begin{equation}
\label{parameters a_i^j components}
a_i^{[j]}=\frac{i+j}{r+1} \text{ if } 1\leq i\leq r-j
\quad\text{and}\quad
a_i^{[j]}=\frac{i+j+1}{r+1} \text{ if } r-j+1\leq i\leq r.
\end{equation}

\item 
$\seq{P_n(x)}$ can be represented via terminating hypergeometric series by
\begin{equation}
\label{1Fr sym MOP}
P_{(r+1)n+j}(x) 
=\frac{(-1)^n\pochhammer[(r+1)n]{j+1}}{n!\,(r+1)^{(r+1)n}}\,x^j\,
\Hypergeometric[x^{r+1}]{1}{r}{-n}{a_1^{[j]},\cdots,a_r^{[j]}}
\quad\text{for all }n\in\N\text{ and }0\leq j\leq r.
\end{equation}

\item 
$\seq{P_n(x)}$ is an Appell sequence, i.e., $\seq{P_n(x)}$ satisfies \eqref{Appell property}, and any Appell $(r+1)$-fold symmetric $r$-orthogonal polynomial sequence is of the form $\dis\seq{c^{-n}P_n(cx)}$ for some $c\neq 0$.
\end{enumerate}
\end{theorem}

For $r=1$, $\seq{P_n(x)}$ are the classical Hermite orthogonal polynomials, 
\eqref{density of the orthogonality measures on the star Appell case} reduces to $\mathrm{d}\hat{\mu}_1(z)=\e^{-z^2}\mathrm{d}z$ for all $z\in\R$,
and \eqref{1Fr sym MOP} gives the quadratic decomposition of the Hermite polynomials in Laguerre polynomials.
For $r=2$, the sequence $\seq{P_n(x)}$ and the components of its cubic decomposition were introduced in \cite{DouakandMaroniClassiquesDeDimensionDeux} and further characterised in \cite[\S 5]{BenCheikhDouak2OPBatemanFunction} and \cite[\S 3.1]{AnaWalter3FoldSym}.

\begin{proof}
For $0\leq j\leq r$, let $a_1^{[j]},\cdots,a_r^{[j]}$ be given by \eqref{parameters a_i^j components} and $a_{r+1}^{[j]}=1$. Moreover, let:
\begin{itemize}
\item 
$a_i^{[(r+1)n+j]}=a_i^{[j]}+n$ for any $1\leq i\leq r+1$, $0\leq j\leq r$, and $n\in\N$,
\hfill\refstepcounter{equation}\textup{(\theequation)}\label{?}

\item 
$\hat{a}_i^{[k]}=\dfrac{i}{r+1}+\ceil{\dfrac{k+1-i}{r+1}}$ for any $1\leq i\leq r+1$ and $k\in\N$.
\hfill\refstepcounter{equation}\textup{(\theequation)}\label{??}
\end{itemize}

Then, for any $1\leq i\leq r+1$ and $0\leq j\leq r$,
\begin{equation}
\hat{a}_i^{[j]}
=\frac{i}{r+1}+\ceil{\dfrac{k+1-i}{r+1}}=
\begin{cases}
\dfrac{i+r+1}{r+1}=a_{i+r-j}^{[j]} & \text{for } 1\leq i\leq j, \vspace*{0,1 cm}\\
\hfil \dfrac{i}{r+1}=a_{i-j}^{[j]} & \text{for } j+1\leq i\leq r+1.
\end{cases}
\end{equation}
Moreover, it is clear that 
$\hat{a}_i^{[(r+1)n+j]}=\hat{a}_i^{[j]}+n$ for any $1\leq i\leq r+1$, $0\leq j\leq r$, and $n\in\N$.
Hence,
\begin{equation}
\left\{a_1^{[j]},\cdots,a_r^{[j]}\right\}=\left\{\hat{a}_1^{[j]},\cdots,\hat{a}_r^{[j]}\right\}
\quad\text{and}\quad
a_{r+1}^{[j]}=\hat{a}_{r+1}^{[j]}
\quad\text{for any }j\in\N.
\end{equation}

Therefore, using \cite[Cor.~5.7]{HypergeometricMOP+BCF},
\begin{equation}
\prod_{i=1}^{r}\pochhammer{a_i^{[j]}}
=\prod_{i=1}^{r}\pochhammer{\hat{a}_i^{[j]}}
=\modifiedStieltjesRogersPoly[r]{n}{j}{\alpha}
\quad\text{for all }n\in\N\text{ and }0\leq j\leq r,
\end{equation}
with
\begin{equation}
\label{alphas Appell case}
\alpha_k
=\prod_{\substack{i=1\\i\not\equiv_{r+1}k}}^{r+1}\hat{a}_i^{[k]}
=\prod_{i=1}^{r}\frac{k+i}{r+1}
=\frac{\pochhammer[r]{k+1}}{(r+1)^r}>0
\quad\text{for all }k\in\N.
\end{equation}
Observe that these are the coefficients of the recurrence relation \eqref{recurrence relation symmetric r-OP hypergeometric case} satisfied by $\seq{P_n(x)}$.

Let $\left(\mu_1,\cdots,\mu_r\right)$ be the system of measures on the positive real line with densities
\begin{equation}
\label{density of the orthogonality measures on R+ Appell case}
\mathrm{d}\mu_j(x)
=\frac{1}{\prod\limits_{i=1}^{r}\Gamma\left(a_i^{[j-1]}\right)}\,
\MeijerG{r,0}{0,r}{-}{a_1^{[j-1]},\cdots,a_r^{[j-1]}}\,\frac{\mathrm{d}x}{x}
\quad\text{for all }1\leq j\leq r\text{ and }x\in\R^+.
\end{equation}
Using \eqref{Mellin transform of a Meijer G-function}, we have, for any $n\in\N$ and $0\leq j\leq r$,
\begin{equation}
\int_{0}^{\infty}x^n\mathrm{d}\mu_j(x)
=\frac{\prod\limits_{i=1}^{r}\Gamma\left(a_i^{[j-1]}+n\right)}{\prod\limits_{i=1}^{r}\Gamma\left(a_i^{[j-1]}\right)}
=\prod_{i=1}^{r}\pochhammer{a_i^{[j-1]}}
=\modifiedStieltjesRogersPoly[r]{n}{j-1}{\alpha}
\text{ for all }n\in\N\text{ and }1\leq j\leq r.
\end{equation}
As a result, using Theorem \ref{multiple orthogonality measures on the (r+1)-star th.*}, we find that $\seq{P_n(x)}$ is $r$-orthogonal with respect to the system of measures $\left(\hat{\mu}_1,\cdots,\hat{\mu}_r\right)$ on $\mathrm{St}^{(r+1)}$, whose densities are given by 
\begin{equation}
\mathrm{d}\hat{\mu}_j(z)=z^{1-j}\,\mathrm{d}\mu_j\left(z^{r+1}\right)
\quad\text{for all }1\leq j\leq r\text{ and }z\in\mathrm{St}^{(r+1)},
\end{equation} 
with $\left(\mu_1,\cdots,\mu_r\right)$ determined by \eqref{density of the orthogonality measures on R+ Appell case}. 
This is equivalent to \eqref{density of the orthogonality measures on the star Appell case}, so we proved part (a) of the theorem.

Furthermore, using Theorem \ref{orthogonality functionals and measures for the components}, the components $\seq{P_n^{[j]}(x)}$, with $0\leq j\leq r$, of the decomposition \eqref{r+1 fold decomposition of a symmetric r-OP} of $\seq{P_n(x)}$ are $r$-orthogonal with respect to the systems of measures $\left(\mu_1^{[j]},\cdots,\mu_r^{[j]}\right)$ on $\R^+$ such that
\begin{equation}
\mathrm{d}\mu_i^{[j]}(x)=\mathrm{d}\mu_{j+i}(x)
=\MeijerG{r,0}{0,r}{-}{a_1^{[j+i-1]},\cdots,a_r^{[j+i-1]}}\,\frac{\mathrm{d}x}{x}
\quad\text{for}\quad 1\leq i\leq r-j,
\end{equation}
and, using \eqref{multiplication by x of a Meijer G-function} and $\hat{a}_k^{[i-1]}+1=\hat{a}_k^{[r+i]}$ for any $1\leq i,k\leq r$,
\begin{equation}
\mathrm{d}\mu_{r-j+i}^{[j]}(x)=x\,\mathrm{d}\mu_i(x)
=\MeijerG{r,0}{0,r}{-}{a_1^{[r+i]},\cdots,a_r^{[r+i]}}\,\frac{\mathrm{d}x}{x}
\quad\text{for}\quad 1\leq i\leq j.
\end{equation}

Therefore, using \cite[Thm.~8.1]{HypergeometricMOP+BCF}, we find that 
\begin{equation}
P_n^{[j]}(x) 
=(-1)^n\prod_{i=1}^{r}\pochhammer{a_i^{[j]}}\,
\Hypergeometric{1}{r}{-n}{a_1^{[j]},\cdots,a_r^{[j]}}
\quad\text{for all }n\in\N\text{ and }0\leq j\leq r.
\end{equation}
Moreover, for any $n\in\N$ and $0\leq j\leq r$,
\begin{equation}
\prod_{i=1}^{r}\pochhammer{a_i^{[j]}}
=\frac{1}{n!}\prod_{i=1}^{r+1}\pochhammer{\frac{i+j}{r+1}}
=\frac{1}{n!}\prod_{i=1}^{r+1}\prod_{k=0}^{n-1}\frac{(r+1)k+i+j}{r+1}
=\frac{\pochhammer[(r+1)n]{j+1}}{n!(r+1)^{(r+1)n}}.
\end{equation}
As a result, using \eqref{r+1 fold decomposition of a symmetric r-OP}, we derive \eqref{1Fr sym MOP}.
Hence, we proved part (b) of the theorem.

Applying the definition of the hypergeometric series \eqref{hypergeometric series} to \eqref{1Fr sym MOP}, we obtain
\begin{equation}
\label{explicit expression Appell sym MOP}
P_{(r+1)n+j}(x)
=\sum_{k=0}^{n}c_{n,k}^{[j]}\,x^{(r+1)k+j}
\quad\text{for all }n\in\N\text{ and }0\leq j\leq r,
\end{equation}
with
\begin{equation}
c_{n,k}^{[j]}
=\frac{(-1)^n\pochhammer[(r+1)n]{j+1}}{n!\,(r+1)^{(r+1)n}}\, \frac{\pochhammer[k]{-n}\,(r+1)^{(r+1)k}}{\pochhammer[(r+1)k]{j+1}}
=\frac{(-1)^{n-k}\,\pochhammer[(r+1)(n-k)]{(r+1)k+j+1}}{(r+1)^{(r+1)(n-k)}\,(n-k)!}.
\end{equation}
Observe that, for any $n\in\N$ and $0\leq k\leq n$:
\begin{itemize}
\item 
$\big((r+1)k+j\big)c_{n,k}^{[j]}=\big((r+1)n+j\big)c_{n,k}^{[j-1]}$
for $1\leq j\leq r$,

\item 
$(r+1)(k+1)c_{n+1,k+1}^{[0]}=(r+1)(n+1)c_{n,k}^{[r]}$.
\end{itemize}

As a result, it follows immediately from differentiating \eqref{explicit expression Appell sym MOP} that $P_{n+1}^{\,\prime}(x)=(n+1)P_n(x)$ for all $n\in\N$,
i.e., $\seq{P_n(x)}$ is an Appell sequence. 
Based on \cite[Eq.~2.11]{Douakd-OPAppell}, a $(r+1)$-fold symmetric $r$-orthogonal polynomial sequence $\seq{Q_n(x)}$ is Appell if and only if there exists $\alpha\neq 0$ such that $\seq{Q_n(x)}$ satisfies the recurrence relation 
\begin{equation}
\label{recurrence relation symmetric r-OP Appell}
Q_{n+r+1}(x)=x\,Q_{n+r}(x)-\alpha\binom{n+r}{r}\,Q_n(x)
\quad\text{for all }n\in\N.
\end{equation}
Observe that if $\alpha=r!(r+1)^{-r}$, then $\seq{Q_n(x)}$ is equal to $\seq{P_n(x)}$ satisfying \eqref{recurrence relation symmetric r-OP hypergeometric case}.
Hence, $\seq{Q_n(x)}$ satisfies the recurrence relation \eqref{recurrence relation symmetric r-OP Appell} for some $\alpha\neq 0$ if and only if
\begin{equation}
Q_n(x)=c^{-n}P_n(cx)
\quad\text{for all }n\in\N,
\quad\text{with }c=\left(\frac{r!}{\alpha(r+1)^r}\right)^{\frac{1}{r+1}}.
\end{equation}
Therefore, part (c) of the theorem also holds, concluding our proof.
\end{proof}

\end{document}